\documentclass{birkjour}

\usepackage[LCY,T1]{fontenc}

\usepackage[russian,main=english]{babel}

\usepackage{amsfonts}
\usepackage{amscd}
\usepackage{bbm}
\usepackage{color}
\usepackage{todonotes}
\usepackage{amsmath}
\usepackage{float}
\usepackage{amsthm}
\usepackage{amssymb}
\usepackage{mathrsfs}
\usepackage{amstext}
\usepackage{graphicx}
\usepackage{tikz}
\usetikzlibrary{matrix}

\usepackage{verbatim}

\usepackage{mathtools}
\mathtoolsset{showonlyrefs}

\numberwithin{equation}{section}

\theoremstyle{definition}

\theoremstyle{plain}
\newtheorem{satz}{Theorem}[section]
\newtheorem{defi}[satz]{Definition}
\newtheorem{cor}[satz]{Corollary}
\newtheorem{lem}[satz]{Lemma}
\newtheorem{prop}[satz]{Proposition}
\newtheorem{rem}[satz]{Remark}

\newtheorem{theorem}[satz]{Theorem}

\newcommand{\supp}{\operatorname{supp}}

\newcommand{\mix}{{\rm mix}}

\newcommand{\re}{\ensuremath{\mathbb{R}}}

\newcommand{\N}{\ensuremath{\mathbb{N}}}
\newcommand{\n}{\ensuremath{{\N}_0}}
\newcommand{\nd}{\ensuremath{\n^d}}

\newcommand{\zz}{\ensuremath{\mathbb{Z}}}
\newcommand{\Z}{{\ensuremath{\zz}^d}}

\newcommand{\C}{\ensuremath{\mathbb{C}}}

\newcommand{\tor}{\ensuremath{\mathbb{T}}}

\newcommand{\ca}{\ensuremath{\mathcal A}}

\newcommand{\cF}{\ensuremath{\mathcal F}}

\newcommand{\bk}{\ensuremath{\mathbf k}}

\newcommand{\bj}{\ensuremath{\mathbf j}}
\newcommand{\bn}{\ensuremath{\mathbf n}}

\newcommand{\bx}{\ensuremath{\mathbf x}}
\newcommand{\bX}{\ensuremath{\mathbf X}}
\newcommand{\by}{\ensuremath{\mathbf y}}
\newcommand{\bc}{\ensuremath{\mathbf c}}

\DeclareMathOperator*{\argmin}{arg\,min}

\newcommand{\hyp}{\operatorname{%
\mathchoice{%
\includegraphics{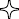}}{%
\includegraphics{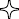}}{%
\includegraphics{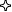}}{%
\includegraphics{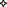}}
}}

\newcommand{\bproof}{\begin{proof}}
\newcommand{\eproof}{\end{proof}}
\renewcommand{\proofname}{\normalfont{\textbf{Proof}}}

\renewcommand{\theenumi}{\roman{enumi}}
\renewcommand{\labelenumi}{{\rm (\theenumi)}}

\newlength{\fixboxwidth}
\newcommand{\be}{\begin{equation}}
\newcommand{\ee}{\end{equation}}
\newcommand{\beq}{\begin{eqnarray}}
\newcommand{\beqq}{\begin{eqnarray*}}
\newcommand{\eeq}{\end{eqnarray}}
\newcommand{\eeqq}{\end{eqnarray*}}

\makeatletter
\def\blfootnote{\xdef\@thefnmark{}\@footnotetext}
\makeatother

\usepackage{hyperref}

\hypersetup{pdftitle={Best m-term approximation of weighted Wiener classes},
pdfpagemode=UseOutlines,
pdfauthor={Moritz Moeller and Serhii Stasyuk and Tino Ullrich},
pdfsubject={Mathematics},
unicode=true}

\begin{document}
\title[Best $m$-term approximation in weighted Wiener spaces]{Best $m$-term trigonometric approximation in weighted Wiener spaces and applications}
\author[Moritz Moeller]{Moritz Moeller}
\address{%
Chemnitz University of Technology \\
Faculty of Mathematics}
\author[Serhii Stasyuk]{Serhii Stasyuk}
\address{%
Chemnitz University of Technology \\
Faculty of Mathematics\\
Institute of Mathematics of NAS of Ukraine}
\author[Tino Ullrich]{Tino Ullrich}
\address{%
Chemnitz University of Technology \\
Faculty of Mathematics}
\email{tino.ullrich@math.tu-chemnitz.de}



\begin{abstract}

In this paper we study best \(m\)-term trigonometric approximation in weighted Wiener spaces and {its} consequences for Besov and Sobolev spaces with bounded mixed derivative/difference. We obtain several sharp asymptotic bounds for weighted Wiener spaces  {including the quasi-Banach case}. It has  {recently been observed} that best \(m\)-term trigonometric widths in the uniform norm together with recovery algorithms stemming from compressed sensing serve to control the optimal sampling recovery error in various relevant spaces of multivariate functions. We use a collection of old and new tools as well as novel findings to extend the recovery bounds to classical multivariate smoothness  {spaces.} It turns out that embeddings into Wiener spaces serve as a powerful tool to improve certain recent bounds.
\end{abstract}

\maketitle

\section{Introduction}

There is growing interest in estimates for the best $m$-term trigonometric approximation errors of multivariate functions in the uniform or Wiener norm.  {These approximation error bounds serve} as upper bounds for (non-linear) sampling recovery errors \cite{JUV23}, \cite{MPU25_1} in $L_q$.   
We study multivariate weighted Wiener spaces \(S^{r}_{\theta} \ca\),  {see Definition \ref{Wmix},} with mixed smoothness $r>0$ and summation index $0<\theta<\infty$. The spaces of interest can be seen as weighted variants of the classical Wiener algebra $\ca$, which are compactly embedded into $L_\infty$ on the $d$-torus. They can be seen as periodic versions of so-called Barron spaces \cite{Ba93}, \cite{Vo22}, see also Wermer \cite{We54}, and share similar properties regarding non-linear approximation. 
Recently,  {several authors studied their approximation theoretic properties, see Chen and Jiang \cite{CJ24} to begin with. DeVore, Petrova and Wojtaszczyk \cite{DPW24} worked on the best \(m\)-term trigonometric approximation for generalized Wiener spaces. A first relation between best \(m\)-term approximation and sampling numbers including first  results for mixed Wiener spaces are presented in Jahn, T. Ullrich and Voigtlaender \cite{JUV23}. Kolomoitsev, Lomako and Tikhonov \cite{KLT23} worked on linear approximation for Wiener spaces. Tractability issues for high-dimensional approximation have been studied by Krieg \cite{Kri23} and Moeller, Stasyuk, T. Ullrich \cite{MSU24_1}. Gelfand numbers and best \(m\)-term approximation for weighted Wiener spaces have been adressed by Moeller \cite{Mo23}. Instance optimality and further recovery methods are presented in Moeller, Pozharska and T.Ullrich \cite{MPU25_1}. Several $s$-numbers in some mixed spaces including mixed Wiener spaces play a role in V.D. Nguyen, V.K. Nguyen and Sickel \cite{NNS22}. Finally, we would like to refer to  Stepanets \cite{Stepanets_UMZh_2001_N8}, \cite[Chapter 11]{Stepanets_MAT_2005} for early results on generalized Wiener spaces.} 

In particular, we are interested in bounds on the best $n$-term trigonometric approximation errors appearing in the results in \cite{JUV23, MPU25_1} which yield upper bounds on the (non-linear) sampling widths in \(L_q\), $2\leq q\leq \infty$. In the very recent contributions \cite{MPU25_1} the authors showed some refinements of these results that allow for bounding the sampling recovery error for any instance $f$ by the error of best $m$-term approximation of this $f$ (instance optimality).  {In this paper we contribute another result of this type, where the target space is the Wiener Algebra \(\ca\) instead of \(L_q\). Results of these type have been studied in \cite{Kri23} and we improve on them, by showing an instance optimal version. For \(m \geq Cnd \log^\ast(n)^2 \log^\ast(M)\) iid random samples Theorem \ref{samp_A1} states that with high probability it holds
\begin{equation} 
 \| f - R_{m}\big( f;\bX)\|_{L_q} \leq C n^{1/2-1/q}\left(n^{-1/2} \sigma_n\big(f\big)_{{\ca} } + E_{[-M,M]^d}\big(f\big)_{L_\infty}\right)\,.
\end{equation} }

The main strategy to prove bounds for best $n$-term trigonometric approximation errors for weighted Wiener spaces is a classical decomposition strategy. Starting from a decomposition according to frequencies with roughly equal weights we then employ results for unweighted Wiener spaces. Throughout this paper we will make use of results that we obtained in our previous paper \cite{MSU24_1} where we dealt with unweighted Wiener spaces and their embeddings, see also Temlyakov \cite{Tem15}. 

The first result shown with this strategy, Theorem \ref{sigma_m_A_0<p<infty,q_bigger=_2}, gives, for $0<\theta\leq\infty$, $2\leq q < \infty$ and  $r>(1-\frac{1}{\theta})_+$, the asymptotic bound 
 $$
\sigma_m(S^{r}_{\theta}\mathcal{A})_{L_q}
 \asymp m^{-(r+{\frac{1}{\theta}-\frac{1}{2}})} \log^\ast(m)^{(d-1)r},
 $$
where in the case of \(q = \infty\) we get an additional \(\sqrt{\log^\ast(m)}\) in the upper bound (with the same lower bound as before). Similar results of this type were already shown for mixed Wiener spaces in \cite{JUV23}, however, they only work under the artificial condition \(r>1/2\) due to the techniques used in those proofs (embedding into Sobolev \(L_2\) space first). Here, with the groundwork from \cite{MSU24_1}, we can get this for $r>(1-1/\theta)_+$, and specifically if \(\theta \leq 1\), i.e.\ in the setting of mixed Wiener spaces, we only have the trivial condition \(r>0\). Results for related isotropic spaces have been obtained by DeVore and Temlyakov \cite{DeTe95}.

In addition, motivated by Theorem \ref{samp_A1}, we  {consider embeddings of weighted into unweighted} Wiener spaces. Theorem  {\ref{A2A}} states, for  {\( \theta, \eta \in (0,\infty]\)} and \(r > \big(\frac{1}{\eta} - \frac{1}{\theta}\big)_+\),
\be
\sigma_m(S^r_\theta\ca)_{\ca_\eta} \asymp m^{\frac{1}{\eta} - \frac{1}{\theta}-r} \log^\ast(m)^{(d-1)r},
\ee
 where we measure the error in the \(\eta\)-Wiener space 
 $\mathcal{A}_{\eta}=S^{0}_{\eta}\mathcal{A}$, a simple modification of the classical Wiener space (see Definition \ref{Wmix}). As a consequence, using Theorem  {\ref{samp_A1}}, the result above supports the formulation of upper bounds on the non-linear sampling widths.
 {Corollary \ref{nlsn_up_good}} states that for \(0 < \theta \leq \infty\), $2\leq q\leq \infty$ and \( r > (1 - 1/\theta)_+ \) it holds
\begin{equation} 
\varrho_{m \log(m)^3 }(S^{r}_{\theta}\mathcal{A})_{L_q}  \lesssim m^{-( r+{\frac{1}{\theta}+\frac{1}{q}-1})} \log^\ast(m)^{(d-1) r}.
\end{equation}

For \(q = 2\) this asymptotic bound is sharp 
 up to a power of the logarithmic factor 
 that is independent of \(d\) otherwise at least the main rate is known to be sharp, see \cite[Remark 1]{MPU25_1}. Indeed, lower bounds on the linear and non-linear sampling widths in weighted Wiener spaces can be shown using the natural lower bounds obtained by linear and Gelfand widths, see Pietsch \cite{Pie87} for an axiomatic approach. The Kolmogorov and Gelfand widths in this setting were studied in \cite{NNS22}.  {The lower bound on the linear sampling widths via the Kolmogorov/linear widths (\cite[Theorem 4.6 (ii)]{NNS22}) is
$$
\varrho^{\operatorname{lin}}_m\big(S^{r}_{1}\mathcal{A}\big)_{L_2} \geq d_m\big(S^{r}_{1}\mathcal{A}\big)_{L_2} \gtrsim \ m^{- r} \log^\ast(m)^{(d-1) r } .$$ 
In fact, Kolmogorov widths in $L_2$ equal the linear widths which are a trivial lower bound for the linear sampling  {widths.} This shows that the error in non-linear sampling (for \(q = 2\) and \(\theta =1\)) decays faster by \(\frac{1}{2}\) in the main rate than the error for linear sampling. }

In addition, we study  embeddings of these weighted Wiener spaces into dominating mixed smoothness Besov-Sobolev spaces  {(and vice versa)}, so that our results can be transferred between these two settings. 
Our first result here, Theorem \ref{SB_subset_SA_theta<2<p}, embeds Besov spaces into Wiener spaces for $0<\theta \leq 2\leq p<\infty$ and $r\geq 0$ that is
 \begin{equation}
 S^{r+\frac{1}{\theta}-\frac{1}{2}}_{p,\theta}B
 \hookrightarrow 
 S^{r}_{\theta}\mathcal{A}\,,
 \end{equation}
where the embedding has norm $1$.  

Our second result, Theorem \ref{SB+1/p_subset_SA}, also embeds Besov spaces into Wiener spaces for  {$1<p\leq 2$, $0<\theta\leq 2$} and \(r \geq 0\), that is
 \begin{equation}
 S^{r+{\frac{1}{p}}+{\frac{1}{\theta}-1}}_{p,\theta}B
 \hookrightarrow S^{r}_{\theta}\mathcal{A} .
 \end{equation} 

On the other hand, if  {$2\leq p<\infty$, $2\leq \theta\leq\infty$} and \( r > 0 \) we show the reverse embedding in Theorem \ref{SA^-1/p_subset_SB_2<p,theta<infty}
 \begin{equation} 
 S^{r+1-\frac{1}{p}-\frac{1}{\theta}}_{\theta}\mathcal{A}
 \hookrightarrow S^r_{p,\theta}B .
 \end{equation}
With the aid of these three inclusions together with the Jawerth-Franke embedding (see \cite[Lem.\ 3.4.3]{DTU18} and \cite[Rem.\ 3.4.4]{DTU18}) we are also able to establish embeddings between weighted Wiener spaces and Sobolev spaces with mixed smoothness
  \be  
 S^{r+\frac{2}{p}-1}_{p}W
 \hookrightarrow S^{r}_{p}\mathcal{A} \,, \ 
 1<p\leq 2,
 \ee  
 \be 
 S^{r+1-\frac{2}{p}}_{p}\mathcal{A}
 \hookrightarrow S^r_{p}W \,, \ 
 2\leq p<\infty,
 \ee  
 for $r\geq 0$ in Theorems \ref{SW+2/p_subset_SA_p} and  \ref{SA-2/p_subset_SW_p}, consequently.  
Moreover, we can use  {some of} these embeddings to establish sharp bounds on the studied quantities of Besov-Sobolev spaces with mixed smoothness in the Wiener norm.

Theorem \ref{sigma_m_Besov_in_Wiener_q<theta}  
states for 
  {$1<p\leq 2$, $1\leq \theta\leq 2$} 
 and \(r > \frac{1}{p}\)
 that it holds 
 $$ 
   \sigma_{m}(S^{r}_{p,\theta}B)_{\mathcal{A}} \asymp
   m^{-(r -\frac{1}{p})}
   \log^\ast(m)^{(d-1)(r + 1 -\frac{1}{p} -\frac{1}{\theta})}  .
 $$ 
In Theorem \ref{WinA_q<p} we  {show} for $1<p\leq 2$ and $r>\frac{1}{p}$ that it holds
 $$ 
   \sigma_{m}(S^{r}_{p}W)_{\mathcal{A}} \asymp
   m^{-(r -\frac{1}{p} )}
   \log^\ast(m)^{(d-1)(r+1 -\frac{2}{p})}.
 $$

\vspace{5pt}
{\bf Notation.} As usual, $\N$ denotes the natural numbers, $\N_0=\N\cup\{0\}$,
$\zz$ denotes the integers,
$\re$ the real numbers, and $\C$ the complex numbers. The letter $d$ is always reserved for the underlying dimension in $\re^d$, $\zz^d$, etc. For $a\in \re$ we denote $a_+ \coloneqq \max\{a,0\}$, for \(\bx \in \re^d\) we define this pointwise as \( \bx_+ \coloneqq ((x_1)_+,\ldots,  (x_d)_+)\), and by $ \bx = (x_1,\ldots,x_d)>0$ we mean that each coordinate is positive.
By \(\bx \by \) we indicate the  {usual} inner product. If $X$ and $Y$ are two (quasi-)normed spaces, the (quasi-)norm of an element $x$  in $X$ will be denoted by 
 \(\Vert x\Vert_X\). We use the usual notations \(L_p\) for the Lebesgue \mbox{(quasi-)}Banach spaces and \(\ell_p\) for the sequence spaces, including the slightly less common notation \(\ell_0\), with \(\Vert \bx \Vert_{\ell_0} \coloneqq \# \supp{\bx}\).
 The notation $X \hookrightarrow Y$ indicates that the identity operator is continuous.  
 For two sequences $a_n$ and $b_n$ we will write
$a_n \lesssim b_n$ if there exists a constant $c>0$ and \(n_0 \in \N \) such that $a_n \leq c\,b_n$
for all $n \geq n_0$. We will write $a_n \asymp b_n$ if $a_n \lesssim b_n$ and $b_n\lesssim a_n$.  
By \(\tor\) we indicate the torus \( [0,1)\) where we identify the endpoints. The set of continuous functions  {(on $d$-dimensional torus $\tor^d$) is denoted by \(C(\tor^d)\).} We denote by $f\in \mathcal{T}([-M,M]^d)$ that \(f\) is a trigonometric polynomial with support on the frequencies in \(\zz^d\cap [-M,M]^d\)  {(i.e \(f(\bx) = \sum_{\bk \in \Z \cap [-M,M]^d} a_\bk \exp(2\pi \mathrm{i} \bk \bx) \))}. For a set \(A\) we denote its cardinality by \(\#A\). When we write a logarithm 
 \(\log^\ast(m)\) we mean \(\max\big\{ \log_2(m), 1\big\}\).
 
\section{Weighted Wiener spaces}

Before we discuss weighted Wiener spaces, let us briefly introduce the classical Wiener algebra.

\begin{defi}
For $f\in L_1(\mathbb{T}^d)$ we define the norm of  {\(f\) in} the  Wiener Algebra as the sum of absolute values of the Fourier coefficients
 $$
\|f\|_{\ca}
 := \|f\|_{\ca(\mathbb{T}^d)}
 := \sum_{\bk\in \mathbb{Z}^d} |\hat{f}(\bk)| \,
 $$
 with
 $$
\hat{f}(\bk)=\int_{\tor^d}f(\bx)\exp(-2\pi \mathrm{i} \bk\bx)\,\mathrm{d}\bx\quad,\quad \bk \in \mathbb{Z}^d\,.
 $$
 \noindent
 We  {define} the Wiener space \(\ca\) to be the space of all functions from \(L_1(\tor^d)\) where this norm is finite.
 \end{defi}

 \begin{defi}[Weighted (mixed) Wiener space] \label{Wmix}
For \( r \geq 0\) and \(0 < \theta < \infty\) we define the (quasi-)norm 
 \[ \|f\|_{S^{r}_{\theta} \ca} \coloneqq \Big( 
\sum_{\bk\in\Z} \prod_{i=1}^d (1+|k_i|)^{ r\theta}|\hat{f}(\bk)|^\theta \Big)^\frac{1}{\theta}. \]
The corresponding space
 \begin{equation}
S^{r}_{\theta}\ca(\tor^d) \coloneqq S^{r}_{\theta}\ca \coloneqq \Big\{ f \in L_1(\tor^d)~:~\Vert f\Vert_{S^{r}_{\theta} \ca} < \infty \Big\}  
 \end{equation}
is called weighted (quasi-)Wiener space of functions with mixed smoothness.
In the case \(\theta = \infty\) the usual modification is made and the space is called Korobov space (see e.g. \cite[Section 3.3]{DTU18}). The space \(S^0_1\ca\) is simply the (original) Wiener space \(\ca\), and $\ca_{\theta} :=S^{0}_{\theta} \ca$. The mixed Wiener spaces \(S_1^r \ca\) are often referred to in the literature as \(\ca^r_\mix\).
 \end{defi}

\begin{rem}\label{embed}
While the space \(S^{r}_{\theta} \ca\) is a function space, it can also be embedded into \(\ell_{\theta}\) quite naturally for all \(\theta > 0\) by the following  {isometric} operator 
 $$
 A_r f= \bigg(\prod_{i=1}^d (1+|k_i|)^r \hat{f}(\bk)\bigg)_{\bk\in \zz^d}.
 $$
 This also immediately enables the usage of the Stechkin Lemma \ref{stechkin} for\hspace{-0.1pt} Wiener spaces.
\end{rem}

\section{Several widths}

We start by defining several approximation widths. These quantities are closely related to \(s\)-numbers and share many of their properties, see e.g. \cite{Pie87}. Our main focus lies on best \(m\)-term trigonometric approximation widths which probably first appeared in Pietsch \cite{Pie81}. The best \(m\)-term approximation did return to the forefront of research when the theory of compressed sensing evolved. The recent papers \cite{JUV23, Kri23,  MPU25_1} linked these widths to errors of function recovery from samples. 

\subsection{Trigonometric best $m$-term approximation widths}

\begin{defi}[Best \(m\)-term approximation width] \label{def:bm} Let $X$ be a quasi-Banach space and $\mathcal{D} \subset X$ a dictionary.

{\em (i)} For $f\in X$  we
define its best \(m\)-term  approximation error with respect to the dictionary $\mathcal{D}$ as 
\[ \sigma_m(f)_X \coloneqq \sigma_m(f,\mathcal{D})_X \coloneqq\inf_{s \in \Sigma_{m}}  \|f - s\|_X, \]
where \(\Sigma_m\) is the set of all linear combinations of at most \(m\) elements of the dictionary \(\mathcal{D}\).

{\em (ii)} Let \(\cF \hookrightarrow X\) denote a further quasi-Banach space continuously embedded into $X$. The best $m$-term approximation width of \(\cF \hookrightarrow X\) is defined accordingly as
\[ \sigma_m(\cF)_X \coloneqq \sup_{ \Vert f \Vert_\cF \leq 1} \sigma_m(f)_X.\]
\end{defi}
Note, the best \(m\)-term approximation widths satisfy a multiplicative property (typical for \(s\)-numbers, although they are not $s$-numbers), see e.g. \cite[Lemma 2.8]{Mo23}. For spaces \(X,Y,Z\) with corresponding embeddings the following relation holds
\begin{equation}\label{mul_snum}
    \sigma_{m_1+m_2}(X)_Z \leq \sigma_{m_1}(X)_Y \sigma_{m_2}(Y)_Z.
\end{equation}

In this paper we essentially deal with two different, however related, situations. Our main interest is the multivariate trigonometric dictionary given by $\mathcal{T}^d = \{\exp(2\pi \mathrm{i} \bk\cdot)~:~\bk \in \mathbb{Z}^d\}$\,, such that  
$$\Sigma_m := \Big\{\sum\limits_{\bk \in \Lambda_m} a_\bk \exp(2 \pi \mathrm{i} \bk\cdot)~:~\# \Lambda_m \leq m, a_{\bk} \in \mathbb{C}\Big\}\,.
$$ 
On the other hand we consider a discrete situation for sequences $\bx \in \ell_p$, where
$$
    \sigma_m(\bx)_X := \inf\limits_{\|\by\|_{\ell_0}\leq m} \|\bx-\by\|_X
$$
fitting also to the framework from Definition \ref{def:bm}. Note, that the corresponding dictionary is given by the set of all sequences with a single nonzero entry. A classical result for the best \(m\)-term approximation error measured in $\ell_q$ is often referred to as  { the Stechkin} Lemma, see \cite[Lemma 7.4.2]{DTU18} and the references there.

\begin{lem} \label{stechkin}
Let \(0 < p <q \leq \infty \). Then it holds for all \(\bx \in \ell_p\) 
\begin{equation}
\sigma_m(\bx)_{\ell_q} \leq (m+1)^{\frac{1}{q}-\frac{1}{p}} \|\bx\|_{\ell_p}\,. 
\end{equation}
\end{lem}

Apart from best $m$-term, i.e. nonlinear approximation, we need the best linear approximation with respect to a given subspace of trigonometric polynomials. 

 \begin{defi}[Error of best trigonometric approximation]\label{defi_taw}
Let \(M \in \N\).  { {The error of best trigonometric approximation of a function}}  {space}   {\( \cF \hookrightarrow L_\infty \)} is defined as
\be
E_{[-M,M]^d}(\cF)_{L_\infty} \hspace{-2pt}
 \coloneqq \hspace{-2pt}  {\sup_{\Vert f \Vert_{\cF} \leq 1} } 
  {E_{[-M,M]^d}\big(f\big)_{L_\infty}
 \coloneqq }\hspace{-2pt}  {\sup_{\Vert f \Vert_{\cF} \leq 1} }\inf_{g \in \mathcal{T}([-M,M]^d)} \Vert f-g \Vert_{L_\infty}\,.
\ee
\end{defi}

\subsection{Sampling widths}
Let us first introduce the (non-linear) sampling  {width}.  {It represents} the optimal worst-case error for  {recovery} using optimal ``standard information'' (point evaluations).
\begin{defi}[Non-linear sampling  {width}] \label{defi_nl_samp}
Let $\mathcal{F}$ be a  {(quasi-)}normed space of functions on $D$, where function evaluations are continuous, which is continuously embedded into a Banach space $Y$. Then the \(m\)-th (generally non-linear) sampling width is defined as
\begin{equation}  \label{defi_samp_eq}
\varrho_m(\mathcal{F})_Y \coloneqq \inf_{t_1, \ldots,  {t_{m}} \in D} \inf_{R\colon {\C^{m}} \to Y} \sup_{\|f\|_\mathcal{F} \leq 1} \big\Vert f - R \big(f(t_1), \ldots, f( {t_{m}})\big)\big\Vert_Y.
\end{equation}
\end{defi}
We also define  {its} linear counterpart.
\begin{defi}[Linear sampling  {width}] \label{defi_l_samp} Let $\mathcal{F}$ again be a  {(quasi-)}normed space of functions on $D$, where function evaluations are continuous, which is continuously embedded into a Banach space $Y$. Then the \(m\)-th linear sampling width is defined as
\begin{equation} 
\varrho^{\operatorname{lin}}_m(\mathcal{F})_Y \coloneqq \inf_{t_1, \ldots, t_{m} \in D} \inf_{\substack{R\colon\C^{m} \to Y \\ {\rm linear} }} \sup_{\|f\|_\mathcal{F} \leq 1} \big\Vert f - R \big(f(t_1), \dots, f(t_{m})\big)\big\Vert_Y.
\end{equation}
Note that \(R\colon\C^{m} \to Y\) is restricted to linear maps here. This is the only difference between the two definitions.
\end{defi}

 For $\bX = \{\bx^1,...,\bx^m\}$ with \(\bx^i \in \tor^d\) for $i=1,\dots,m$, we will write  \[ R_m(f; \bX) \coloneqq R\big(f(\bx^1), \ldots, f(\bx^m)  \big)\] for short.

Jahn, T. Ullrich and Voigtlaender have shown in \cite{JUV23}, based on compressed sensing techniques (basis pursuit denoising), that the sampling widths measured in \(L_2\) can be bounded  {from above} by a sum of the best \(m\)-term approximation  {width} and the  error of the best trigonometric approximation measured in \(L_\infty\).

This result was then refined  {to an instance optimal version for more general \(L_q\), $2\leq q\leq\infty$, spaces instead of only \(L_2\)} in Moeller, Pozharska and T.~Ullrich 
\cite{MPU25_1}. The oversampling needed for this result comes from the  {(}random{)} construction of a restricted isometry Fourier matrix, see \cite{Bour14}, \cite{BDJR21} and \cite{HaRe17}.

\begin{prop}[{\cite[Theorem 3]{MPU25_1}}]\label{samp_Lp}
For any \(M, n \in \N\) there exists a  {universal constant \(C\) and a} recovery operator \(R_m\) using $m$  {iid}  random samples $\bX = \{\bx^1,...,\bx^m\}$ such that uniformly for all \(f \in C(\tor^d)\) and \(2 \leq q \leq \infty\) it holds
\begin{equation} 
  \| f - R_{m}\big( f;\bX)\|_{L_q} \leq C n^{1/2-1/q}\left( \sigma_n\big(f\big)_{L_\infty} + E_{[-M,M]^d}\big(f\big)_{L_\infty}\right)\,,
\end{equation}
with high probability, where
\begin{equation} \label{m_vs_n}
   m \geq \lceil C nd \log^\ast(n)^2 \log^\ast(M)\rceil .
\end{equation}
\end{prop}

We only need the existence of  {a recovery operator} \(R_m\) here,  {that is admissible in Definition \ref{defi_nl_samp}}. Indeed there are two different valid options that even fulfill this bound  {with high probability}, namely  {``square root Lasso''} {\tt (rLasso)}  given in \cite[Definition 1]{MPU25_1}. Here, we  will refine this result by using an even stronger norm, the  {\(\ca\)} norm,  {weighted with a factor of} $n^{-1/2}$. Some difficulties that arise when dealing with the \(L_\infty\) norm can be avoided in this setting, see Remark \ref{sqrtlog}. 

\begin{theorem}\label{samp_A1}
For any \(M, n \in \N\) there exists a  {universal constant \(C\) and a} recovery operator \(R_m\) using $m$  {iid} random samples $\bX = \{\bx^1,...,\bx^m\}$ such that uniformly for all \(f \in C(\tor^d)\) and \(2 \leq q \leq \infty\) it holds
\begin{equation} 
 \| f - R_{m}\big( f;\bX)\|_{L_q} \leq C n^{1/2-1/q}\left(n^{-1/2} \sigma_n\big(f\big)_{  {\ca} } + E_{[-M,M]^d}\big(f\big)_{L_\infty}\right)\, ,
\end{equation}
with high probability, where
\begin{equation}
   m \geq \lceil Cnd \log^\ast(n)^2 \log^\ast(M)\rceil.
\end{equation}
\end{theorem}

   {   
For the proof of Theorem \ref{samp_A1}, we will make use of the following specific version of the de la Vall\'{e}e Poussin operator \(V_M\) introduced in \cite{FT04} and already used for similar results in \cite{JUV23} and \cite{MSU24_1} 
\begin{equation}
V_M (f)({\bx}) =  \sum_{{\bk}\in \Z} \hat{f}({\bk}) \, v_{\bk}\exp(2 \pi \mathrm{i}  {\bk} {\bx}),
\end{equation}
with weights \( v_{\bk}= \prod_{j = 1}^d v_{k_j}\) satisfying 
\begin{equation}
v_{k_j} = 
\begin{cases}
1, & \vert k_j\vert \leq M , \\
\frac{(2d+1)M - \vert k_j\vert}{2dM} \,, & M <\vert k_j\vert \leq (2d+1)M , \\
0 , &\vert k_j\vert > (2d+1)M\,.
\end{cases}
\end{equation}

\begin{lem}\label{prop_dlVP}
 For \(D = M(2d+1)\), the de la Vall\'{e}e Poussin operator \(V_M\) has the following properties.
  \begin{itemize} 
      \item[(i)] \(V_M g \in \mathcal{T}([-D,D]^d)\) for all \(g \in L_1(\tor^d)\).
      \item[(ii)] \(V_M g = g\) for all \(g \in \mathcal{T}([-M,M])^{d}\).
      \item[(iii)] \(V_M g = \sum_{\bk \in \Z \cap [-D,D]^d} a_\bk \hat{g}(\bk) \exp(2\pi \mathrm{i} \bk \cdot)\) with some \(\vert a_\bk \vert \leq 1\) for \(\bk \in \Z \cap [-D,D]^d\). 
      \item[(iv)] \(\Vert V_M \Vert_{L_\infty \to L_\infty} \leq e\), see e.g. \cite[Proof of Theorem 3.1]{MSU24_1}.
  \end{itemize}  
\end{lem}}

\begin{proof}[Proof of Theorem \ref{samp_A1}]
    We follow the proof of \cite[Theorem 3]{MPU25_1} and show only the cases \(q = 2\) and \(q=\infty\), all other cases then follow by interpolation.

    Identify \(\{-D, \ldots, D \}^d = \{\bk_1, \ldots, \bk_N \}\), this implies \(N = (2D+1)^d\). We start as in \cite[Theorem 3]{MPU25_1} with \(q= \infty\) but make a different choice for \(f^\ast\), namely \(f^\ast = V_M(f)\). Let now \(\bc^\ast\) be the coefficient vector of \(f^\ast\), and \(\bc\) the coefficient vector of \(R_m(f,\bX)\) (due to the construction of \(R_m\) {and (i) from Lemma \ref{prop_dlVP}} we have \(\supp \bc  \subset  {[-D,D]}^d\)). This yields for the recovery operator \(R_m\) from \cite[Definition 1]{MPU25_1} with absolute constants \(\beta, \delta \geq 1\) that
\begin{align} \label{boot}
    \begin{split}
        \| f^\ast - R_{m} \big( f;\boldsymbol{{\rm X}}) \|_{L_\infty} & \leq \sum_{j=1}^{N} \big\vert{ {c^\ast _{j}} -  {c _{j}}} \big\vert \| {\exp(2\pi \mathrm{i} \bk_j \cdot)}\|_{L_\infty} \leq  \|\bc^\ast - \bc  \|_{\ell_1} \\
        & \leq  { \beta \sigma_n(\bc^\ast)_{\ell_1} + \delta \cdot  \sqrt{ n} \frac{1}{\sqrt{m}} \Vert (V_M f)(\bX) - f(\bX)\Vert_{\ell_2}}\\
         & \leq  { \beta \sigma_n(\bc^\ast)_{\ell_1} + \delta \cdot  \sqrt{ n} \Vert (V_M f)- f\Vert_{L_\infty} }\\
        & = \beta \sigma_n(\bc^\ast)_{\ell_1} + \delta \cdot  \sqrt{ n} \, \|f - f^\ast\|_{L_\infty},\\
    \end{split}
\end{align}
 with high probability and therefore the existence of  {such a} recovery operator  {\(R_m\)} that satisfies this bound. 
 {Crucially, we need to estimate \(\|f - f^\ast\|_{L_\infty}\) from above. We do so by introducing 
\[g = \argmin_{g \in \mathcal{T}([-M,M])^{d}} \Vert f - g \Vert_{L_\infty},\]
and applying the triangle inequality,
\begin{align}
    \begin{split}
        \|f - f^\ast\|_{L_\infty} &\leq  \|f - g\|_{L_\infty} + \|g - f^\ast\|_{L_\infty} \\
        & = E_{[-M,M]^d}\big(f\big)_{L_\infty} + \|V_M(g) - V_M(f)\|_{L_\infty} \\
        & \leq E_{[-M,M]^d}\big(f\big)_{L_\infty} + e \Vert g - f\Vert_{L_\infty} \\
        &= (1+e) E_{[-M,M]^d}\big(f\big)_{L_\infty},
    \end{split}
\end{align} 
  where we used Lemma \ref{prop_dlVP} (ii) and (iv).}
Now simply consider  {that by Lemma \ref{prop_dlVP} (iii), it holds \( \sigma_n(\bc^\ast)_{\ell_1} \leq \sigma_n(f)_{\ca}\) and therefore} 
\begin{align} \label{downtoto}
    \begin{split}
        \| f - R_{m}\big( f;\boldsymbol{{\rm X}})\|_{L_\infty} & \leq \Vert f - f^\ast \Vert_{L_\infty} + \Vert f^\ast -  R_{m}\big( f;\boldsymbol{{\rm X}}) \Vert_{L_\infty} \\
        &\leq \beta \sigma_n(\bc^\ast)_{\ell_1} + (\delta \cdot \sqrt{ n}+1)\|f - f^\ast\|_{L_\infty} \\
        & \leq C \sqrt{n} \Big( n^{-1/2} \,  \sigma_n\big(f\big)_{\ca} + E_{[-M,M]^d}\big(f\big)_{L_\infty} \Big).
    \end{split}
\end{align}

    This concludes the regime where \(q =\infty\). Let us turn to the situation \(q = 2\). Here we can apply  {the Parseval} identity and get, after applying the second part of \cite[Theorem 2]{MPU25_1} rather than the first one,
\begin{align} \label{boot_L_2}
    \begin{split}
     \| f - R_{m}\big( f;\boldsymbol{{\rm X}})\|_{L_2} & \leq \Vert f^\ast -  R_{m}\big( f;\boldsymbol{{\rm X}}) \Vert_{L_2} + \Vert f -f^\ast \Vert_{L_\infty} \\
     & = \|\bc^\ast - \bc  \|_{\ell_2} + \Vert f -f^\ast \Vert_{L_\infty} \\
    &\leq \beta n^{-1/2}\sigma_n(\bc^\ast)_{\ell_1} + (\delta+1) \cdot \|f - f^\ast\|_{L_\infty}\\
    & \leq C \Big( n^{-1/2} \sigma_n\big(f\big)_{\ca} + E_{[-M,M]^d}\big(f\big)_{L_\infty} \Big).
    \end{split}
\end{align}
After using an interpolation argument between \eqref{downtoto} and \eqref{boot_L_2} we now get the desired bound.
\end{proof}

\begin{rem}\label{rem_krieg}
{\em (i)}  { In this paper we only deal with the trigonometric system. However, it is possible to generalize Theorem \ref{samp_A1} to general bounded orthonormal systems. This is possible by combining results from \cite{BDJR21} with results like \cite[Theorem 2]{MPU25_1}. We may avoid the de la Vall\'{e}e Poussin operator by changing from \(E_{[-M,M]^d}(f)_{L_\infty}\) on the right hand side to the error of the projection $\|f-P_{[-M,M]^d}f\|_{L_\infty}$.}

{\rm (ii)} Theorem \ref{samp_A1} is similar to the results shown by Krieg in \cite{Kri23}. Here we have instance optimality in our result, i.e.,  it holds for every individual function \(f\), not only for the supremum over the entire function class.
\end{rem}

\section{Best $m$-term approximation widths of weighted Wiener  {spaces} in $L_\infty$ and $\mathcal{A}$}
In Theorems 4.4 and 4.5 of our recent paper \cite{MSU24_1} we showed bounds on the best \(m\)-term approximation width for unweighted Wiener spaces, both in \(L_q\), \(2 \leq q < \infty\), and \(L_\infty\). Similar but slightly worse results of this type can also be found in \cite{Tem15}. 

We now want to use these results to show a new bound for our weighted Wiener spaces by discretizing these existing results onto layers of the step hyperbolic cross and treating all weights there as equal. This will allow to improve on the results {\cite[Theorem 4.2, 4.5]{Mo23}} and the much older one in \cite[Theorem 6.1]{DeTe95},  {by giving more general (in terms of \(\theta\)) and (almost precise) bounds}.

\begin{theorem} \label{sigma_m_A_0<p<infty,q_bigger=_2}
If  $0<\theta\leq\infty$ and $r>(1-1/\theta)_+$, then it holds for \(2 \leq q < \infty\) 
 (or $2\leq q\leq\infty$ for $d=1$)
 $$ \sigma_m(S^{r}_{\theta}\mathcal{A})_{L_q}
 \asymp m^{-(r+{\frac{1}{\theta}-\frac{1}{2}})} \log^\ast(m)^{(d-1)r},
 $$
 and for \(q = \infty\), \(d > 1\)
 $$
 m^{-(r+{\frac{1}{\theta}-\frac{1}{2}})} \log^\ast(m)^{(d-1)r}
 \lesssim \sigma_m(S^{r}_{\theta}\mathcal{A})_{L_\infty}
 \lesssim m^{-(r+{\frac{1}{\theta}-\frac{1}{2}})} \log^\ast(m)^{(d-1)r+\frac{1}{2}} .
 $$
\end{theorem}

The proof will use two auxiliary Lemmas which we state here.

\begin{lem} \label{geo_sum}
    For arbitrary \(\alpha \in \re \), \( \beta,\gamma > 0\) and \(L \in \N\) it holds
    \begin{equation}
        \sum_{k = L+1}^\infty (k-L)^\alpha k^\beta 2^{-\gamma k} \lesssim L^{\beta} 2^{- \gamma L} .
    \end{equation}
\end{lem}
\begin{proof}
    Since \(L \geq 1\) we have for all \(k \geq L +1\) that \(k \geq L+1 \geq L \frac{2L}{2L-1}\). This implies \(k \leq 2L (k - L)\). From that we now get
\begin{align}
    \begin{split}
         \sum_{k = L+1}^\infty (k-L)^\alpha k^\beta 2^{-\gamma k} & \leq \sum_{k = L+1}^\infty (k-L)^{\alpha+ \beta} (2L)^\beta 2^{-\gamma k} \\
         & = (2L)^\beta 2^{-\gamma L} \sum_{k = L+1}^\infty (k-L)^{\alpha+ \beta} 2^{-\gamma (k-L)} \\
         & = (2L)^\beta 2^{-\gamma L} \sum_{n = 1}^\infty n^{\alpha+ \beta} 2^{-\gamma n} \\
         & \lesssim L^\beta 2^{-\gamma L}.
    \end{split}
\end{align}
\end{proof} 

This second Lemma is an easy consequence of the Hölder inequality.  

\begin{lem} \label{CSI}
    Let \(\bx \in \ell_p\) and \(q \geq p\). Then it holds
    \begin{equation}
        \Vert \bx \Vert_{\ell_p} \leq \Vert \bx \Vert_{\ell_q} \Vert \bx \Vert_{\ell_0}^{\frac{q-p}{q p}} .
    \end{equation}
\end{lem}
\begin{proof}
    We simply apply the Hölder inequality with \(\tau:=q/p \geq 1\) to \(\vert x_k\vert^p\) and \(1\). This yields,
    \begin{align}
        \begin{split}
            \Vert \bx \Vert_{\ell_p} & = \bigg( \sum_{k=1}^\infty \vert x_k \vert^p \, \mathbbm{1}_{k \in \supp \bx} \bigg)^{1/p}\\
            & \leq \Bigg[ \bigg( \sum_{k=1}^\infty ( \vert x_k \vert^p)^{\frac{q}{p}} \bigg)^{\frac{p}{q}} \bigg(\sum_{k=1}^\infty \mathbbm{1}_{k \in \supp \bx}^{\frac{q}{q-p}} \bigg)^{\frac{q-p}{q}} \Bigg]^{\frac{1}{p}}\\
            & = \bigg( \sum_{k=1}^\infty \vert x_k \vert^{q}\bigg)^{\frac{1}{q}} \Vert \bx \Vert_{\ell_0}^{\frac{q-p}{qp}},\\
        \end{split}  
    \end{align}
    where we used the usual notation \(\Vert \bx \Vert_{\ell_0} \coloneqq \# \supp{\bx}\).
\end{proof}

\begin{proof}[Proof of Theorem \ref{sigma_m_A_0<p<infty,q_bigger=_2}]
Let \( 2 \leq n \in \N\) and \(d\geq 2\), the case \(d=1\) is simpler in some regards and not covered here in detail (see \cite[Theorem 4.2]{Mo23} for that). 

Let \( f \in S^{r}_{\theta}\mathcal{A} \) with $\|f\|_{S^{r}_{\theta}\mathcal{A}}\le 1$.
We can decompose \(f\) into 
\be \label{firstsum}
f(\bx) = \sum_{j=0}^\infty f_j(\bx),
\ee
 \noindent
 where 
 \begin{equation}\label{segment}
 f_j(\bx)= \sum_{\bk \in\hyp_j} \hat{f}(\bk) \exp(2\pi \mathrm{i} \bk\bx).
  \end{equation}
  Here, the uniform step hyperbolic layer \(\hyp_j\) is defined as follows. 
  Starting from the segmentation of \(\zz = \bigcup_{i=0}^\infty I_i\) 
with $I_0 = \{0\}$ and, for $i\in \mathbb{N}$,
 $$
 I_i = \big\{k \in \zz ~:~  2^{i-1} \leq \vert k\vert < 2^{i} \big\}\,.
 $$
This gives \(\#I_i = 2^{i}\).
For \( \bk \in \nd \) we write
\be \label{Ik}
    I_{\bk} = I_{k_1} \times \ldots \times I_{k_d}.
\ee
  {Each} of these cuboid clusters contains 
 \be \label{Vol_Ik}
 \#I_{\bk} = 2^{ {\Vert\bk \Vert_{\ell_1}}}  
 \ee
points. 
This partitions \(\Z\) into the above cuboid clusters (\ref{Ik}), that is \(\Z = \bigcup_{\bk \in \nd} I_{\bk} \).
We can now group blocks on the same uniform step hyperbolic layer together 
\begin{equation}\label{square}
 \hyp_j \coloneqq \bigcup_{\bk \in \nd, \, 
  {\Vert \bk \Vert_{\ell_1}} = j}  I_{\bk} ,
\end{equation}
where the contents of every \(I_{\bk}\) appear in exactly one \(\hyp_j, \, j \in \N_0\), meaning \(\Z = \bigcup_{j=0}^\infty \hyp_j\).  {For each \(\bk \in I_\bn\) this construction implies that \(\prod_{l = 1}^d (1 + \vert k_l\vert) \asymp 2^{\Vert \bn \Vert_{\ell_1}}\). And therefore also for \(\bk \in \hyp_j\) that \(\prod_{l = 1}^d (1 + \vert k_l\vert) \asymp 2^j\).}

 {Now we construct \(P = \sum_{k=0}^\infty P_k\), with \(\supp P_k \subset \hyp_k\) such that \(P \in \Sigma_{cm}\) for some constant c}.  {To estimate the error from above we segment the approximation of \(f\) by \(P\)} 
\begin{equation} \label{s}
 \|f - P\|_{L_q} \leq \sum_{k= 0}^L \|f_k - P_k\|_{L_q}  + \! \sum_{k=L+1}^K\|f_k - P_k\|_{L_q} + \!\sum_{k=K+1}^\infty \|f_k\|_{L_q} \eqqcolon S_1 + S_2 + S_3,
\end{equation}
 {with \(L\), \(K\) and  \(P_k\) chosen below}.
To continue, we need to know how many points are in each uniform step hyperbolic layer \(\hyp_j\). First, consider that \(\hyp_j\) can be decomposed into a number of blocks \(I_{\bk}\) containing \(2^{ {\Vert\bk \Vert_{\ell_1}}} = 2^j\) points each. Since these sets are products of dyadic intervals  {we can count them by considering} the number of possibilities to distribute \(j \in \N_0\) to \(d\) different dimensions. In total, we have
\be\label{C_j}
C_j \coloneqq \# \hyp_j = 2^j \binom{j+d-1}{j} \asymp 2^j j^{d-1}.
\ee
\noindent 
Choose now  {for sufficiently large \(n\)}
\begin{equation}\label{L}
L= \big\lceil n - (d-1) \log^\ast (n) \big\rceil > 0
\end{equation}
and
\begin{equation}\label{K}
K =\Big\lceil n\, \frac{r+{\frac{1}{\theta}-\frac{1}{2}}}{r -1 + \frac{1}{\theta}} - (d-1) \log^\ast (n) \Big\rceil.
\end{equation} 
For \(k=0, \dots, L\) we choose \(P_k = f_k\), then obviously \(S_1 = 0\). For \( k = L+1, \ldots , K\) we instead choose trigonometric polynomials \(P_k\) with 
 \begin{equation}\label{m_k}
m_k = \lceil (k-L)^{-2}2^{L} L^{d-1} \rceil
 \end{equation} 
many frequencies. 
Finally, for $k \geq  K+1$, we choose $P_k:=0$. 
Then in total \(P \coloneqq \sum_{k=0}^KP_k\) is a linear combination of at most
\begin{align} \label{ms}
\begin{split}
\sum_{k= 0}^L C_k + \sum_{k=L+1}^Km_k & \lesssim \sum_{k= 0}^L 2^k k^{d-1} + \sum_{k=L+1}^K \lceil (k-L)^{-2} 2^{L} L^{d-1} \rceil \\
& \leq L^{d-1}\sum_{k= 0}^L 2^k + \sum_{k=L+1}^K 1+ (k-L)^{-2} 2^{L} L^{d-1} \\
& \lesssim  2^{L} L^{d-1} + 2^{L} L^{d-1} + K - L\\
& \lesssim  2^L n^{d-1} + n \\
& \lesssim 2^n
\end{split}
\end{align}
trigonometric monomials.

To estimate \(S_2\)  {we employ \cite[Theorem 4.4]{MSU24_1}. This states for \(2 \leq q < \infty\) that there always exists a \(P_k\) with \(m_k\) frequencies} such that 
\begin{equation} \label{ps1}
\|f_k - P_k\|_{L_q} \lesssim m_k^{-\frac{1}{2}} \|f_k\|_{\mathcal{A}}.
\end{equation} 
 The result in \cite{MSU24_1} even states this inequality with a constant depending only on \(q\), however for us the asymptotic inequality is sufficient.
 In the regime where \(q= \infty\) we instead get by  \cite[Theorem 4.5]{MSU24_1} and, considering that \(f_k \in \mathcal{T}([-2^k,2^k]^d) \),
\begin{equation} \label{ps_inf}
\|f_k - P_k\|_{L_\infty} \lesssim m_k^{-\frac{1}{2}} \sqrt{\log^\ast(2^{dk})} \|f_k\|_{\mathcal{A}} \lesssim m_k^{-\frac{1}{2}} k^{\frac{1}{2}} \|f_k\|_{\mathcal{A}}.
\end{equation} 
Now for \(\theta \geq 1\) recall that \(\Vert f_k \Vert_{S^{r}_{\theta}\mathcal{A}} \leq \Vert f \Vert_{S^{r}_{\theta}\mathcal{A}} \leq  1\) and that the support of \(f_k\) is limited to the \(k\)-th step hyperbolic layer and therefore contains (up to a constant) at most \(2^k k^{d-1}\) points. This now gives, by Remark \ref{embed} and Lemma \ref{CSI} the following:
\begin{align} \label{1top}
\begin{split}
\|f_k\|_{\mathcal{A}} & \lesssim 2^{-kr} \Vert f_k\Vert_{S_{1}^r\mathcal{A}} \\
& \lesssim 2^{-kr} \|f_k\|_{S^{r}_{\theta}\mathcal{A}} 2^{k} k^{d-1} \big(2^{k} k^{d-1} \big)^{-\frac{1}{\theta}}\\
&\lesssim  2^{-k r} 2^{k (1 - \frac{1}{\theta})} k^{(d-1)(1-\frac{1}{\theta})}.
\end{split}
\end{align}

 The cases \(\theta < 1\) and \(q = \infty\) will be discussed later. Therefore it is sufficient to estimate \(S_2\) for  { \(1 \leq \theta \leq \infty\) and \(q < \infty\)} by using \eqref{ps1}, \eqref{1top}  {and \eqref{m_k}} as follows:
\begin{align} \label{s1}
\begin{split}
S_2 & \lesssim  \sum_{k=L+1}^Km_k^{-\frac{1}{2}}  2^{-kr} 2^{k (1 - \frac{1}{\theta})} k^{(d-1)(1-\frac{1}{\theta})} \\
& =  {  2^{-L \frac{1}{2}} L^{-\frac{1}{2} (d-1)} \sum_{k=L+1}^K(k-L)\,    2^{-k ( r  - 1 + \frac{1}{\theta})} k^{(d-1)(1-\frac{1}{\theta})} }  \\
& \lesssim  2^{-L( \frac{1}{2} +  r  -1 + \frac{1}{\theta} )} L^{- (\frac{1}{2} - 1 + \frac{1}{\theta}) (d-1)}  \\
& \asymp 2^{-(n - (d-1) \log^\ast (n))({\frac{1}{\theta}-\frac{1}{2}}+ r )} (n - (d-1) \log^\ast (n))^{-({\frac{1}{\theta}-\frac{1}{2}})(d-1)} \\
& \asymp 2^{-n({\frac{1}{\theta}-\frac{1}{2}}+ r )} n^{(d-1)(\frac{1}{\theta}-\frac{1}{2} + r )} n^{-(\frac{1}{\theta}-\frac{1}{2})(d-1)} \\
& = 2^{-n({\frac{1}{\theta}-\frac{1}{2}}+ r )} n^{(d-1) r }.
\end{split}
\end{align}
In the second line we needed that \(r > 1 - \frac{1}{\theta}\) so that \(\gamma = r - 1 + \frac{1}{\theta}> 0\) and we can apply Lemma \ref{geo_sum}. In the case \(q= \infty\) the term \(k^{1/2}\) from \eqref{ps_inf} simply yields an additional \(n^{1/2}\) here.

To estimate \(S_3\) we use the fact that, by \eqref{1top} it holds
\begin{equation*}
\|f_k\|_{L_q} \leq \|f_k\|_{L_\infty} \leq \sum_{\boldsymbol{\ell} \in \zz^d}|\hat{f_k}(\boldsymbol{\ell})| = \|f_k\|_\ca \lesssim 2^{-k (r -1 + \frac{1}{\theta})}  k^{(d-1)(1-\frac{1}{\theta})},
\end{equation*}
and therefore, by Lemma \ref{geo_sum} and the definition of \(K\) in \eqref{K}, we get 
\begin{align} \label{s2}
\begin{split}
S_3 & \lesssim \sum_{k=K+1}^\infty 2^{-k (r -1 + \frac{1}{\theta})} k^{(d-1)(1-\frac{1}{\theta})} \\
& \lesssim 2^{-K (r -1 + \frac{1}{\theta})}  K^{(d-1)(1-\frac{1}{\theta})} \\
& \lesssim 2^{- \big(n \, \frac{r+{\frac{1}{\theta}-\frac{1}{2}}}{ r - 1 + \frac{1}{\theta} } - (d-1) \log^\ast(n) \big) (r -1 + \frac{1}{\theta}) }  n^{(d-1)(1-\frac{1}{\theta})} \\
& = 2^{-n ( r  + {\frac{1}{\theta}-\frac{1}{2}})} n^{(d-1)  r}.
\end{split} 
\end{align}

Combining now \eqref{ms}, \eqref{s1}, \eqref{s2} we can estimate \eqref{s} as follows  {for some constant \(C=C(d,r,\theta) \in \mathbb{N}\),}
\begin{align}\label{s_finish}
\begin{split} 
\sigma_{C 2^{n}}(f)_{L_q} & \leq \|f - P\|_{L_q} \\
& \leq  S_1 + S_2 + S_3\\
& \lesssim 0 + 2^{-n( r +{\frac{1}{\theta}-\frac{1}{2}})} n^{(d-1) r  }+ 2^{-n ( r  + {\frac{1}{\theta}-\frac{1}{2}})} n^{(d-1)  r  } \\
& \lesssim 2^{-n(  r  +{\frac{1}{\theta}-\frac{1}{2}})} n^ {(d-1)r}\,. 
\end{split}
\end{align}
For \(m \geq 4C\) we can now argue with monotonicity. Let \(n = \lfloor \log^\ast (C^{-1} m) \rfloor\). Note that
 \begin{equation*}
 \frac{1}{2} \log^\ast (C^{-1} m)
 \leq \log^\ast (C^{-1} m) -1 
 \leq n 
 \leq \log^\ast (C^{-1} m) .
 \end{equation*}
Moreover, note that \(2^n \leq C^{-1} m\), and hence
\begin{align}
\begin{split}
  \sigma_{m}(S_\theta^r \ca)_{L_q} & = \sigma_{C C^{-1} m}(S_\theta^r \ca)_{L_q} \\
  & \leq \sigma_{C2^n}(S_\theta^r \ca)_{L_q} \\
  & \lesssim  2^{-n(r+{\frac{1}{\theta}-\frac{1}{2}})} n^{(d-1)r} \\
  & \lesssim (C^{-1}m)^{-(r+{\frac{1}{\theta}-\frac{1}{2}})} \log^\ast(C^{-1}m)^{(d-1)r}\\
  & \lesssim m^{-(r+{\frac{1}{\theta}-\frac{1}{2}})} \log^\ast(m)^{(d-1)r},
\end{split}
\end{align} 
where in the case \(q = \infty\), \(d > 1\) the additional \(\sqrt{ \log^\ast(m)}\) appears again.
The case \(m < 4C\) is trivial. 

To show this upper bound now also for \(\theta<1\) we use the result for \(\theta=1\) and argue as follows by using the Stechkin Lemma and relation \eqref{mul_snum}.
\[\sigma_{2m} \big( S^{r}_{\theta}\mathcal{A}\big)_{L_q} \leq  \sigma_{m}\big( S^{r}_{\theta}\mathcal{A} \big)_{S^{r}_{1}\mathcal{A}} \,\sigma_{m} \big( S^{r}_{1}\mathcal{A} \big)_{L_q} \leq  m^{- \frac{1}{\theta} + 1} \sigma_{m}  \big( S^{r}_{1}\mathcal{A} \big)_{L_q}.\] 
Then apply the upper bound for \(\theta=1\) to obtain
\begin{equation} \label{p2}
\sigma_{2m}\big(S^{r}_{\theta} \ca \big)_{L_q} \lesssim m^{-(r +{\frac{1}{\theta}-\frac{1}{2}})}\log^\ast(m)^{(d-1) r }.
\end{equation} 

Let us prove the lower bounds.
Given \(m \in \N\), choose \(n \in \N\) minimal with \(C_n \geq 2m\). Since \( C_n \asymp 2^nn^{d-1} \) by relation \eqref{C_j}, this implies \(m \asymp 2^nn^{d-1} \asymp C_n = \#\hyp_n\).
To get a lower bound we  construct a fooling function:
\begin{equation}
f(\bx) = m^{-( r + \frac{1}{\theta})} \log^\ast(m)^{(d-1)r} \sum_{\bk \in\hyp_n} \exp(2 \pi \mathrm{i} \bk \bx) .
\end{equation}

This function has \(S^{r}_{\theta}\mathcal{A}\)-norm as follows  
\begin{align}  \label{fnormed}
\begin{split}
\|f\|_{S^{r}_{\theta}\mathcal{A}} &\asymp m^{-( r + \frac{1}{\theta})} \log^\ast(m)^{(d-1)r} \Big(\sum_{\bk \in\hyp_n} \big(\prod_{j=1}^d (1 + |k_j|)^ r  \big)^\theta\Big)^\frac{1}{\theta}\\
& \asymp  m^{-( r + \frac{1}{\theta})} \log^\ast(m)^{(d-1)r} \Big(\sum_{\bk \in\hyp_n} 2^{ r  \theta \Vert \bk \Vert_{\ell_1} }\Big)^\frac{1}{\theta}\\
& \asymp m^{-( r + \frac{1}{\theta})} \log^\ast(m)^{(d-1)r} \Big(\sum_{\bk \in\hyp_n} 2^{ r \theta n}\Big)^\frac{1}{\theta}\\
& \asymp m^{-( r + \frac{1}{\theta})} \log^\ast(m)^{(d-1)r} \Big(2^n n^{d-1} 2^{ r \theta n}\Big)^\frac{1}{\theta}\\
& \asymp 1 .
\end{split}
\end{align}

Now, given an arbitrary set \(K \subset \Z\) of cardinality \(\# K \leq m\) and arbitrary coefficients \((a_\bk)_{\bk \in K}\) we set \(g =  m^{-\frac{1}{2}} \sum_{\bk \in \hyp_n \setminus K} \exp(2 \pi \mathrm{i} \bk \bx) \) and \(h =  \sum_{\bk \in K} a_{\bk} \exp(2 \pi \mathrm{i} \bk \bx)\). For these functions we get
 \begin{equation} \label{hup}
\|g\|_{L_2} \leq m^{-\frac{1}{2}} \big(\#\hyp_n\big)^{\frac{1}{2}} \lesssim m^{-\frac{1}{2}} m^{\frac{1}{2}} = 1.
\end{equation}
Moreover, we see that because of \(\# \hyp_n - \# K \geq C_n - m \geq 2m - m = m \) that
\begin{align}  \label{tangle}
\begin{split}
\langle f-h,g \rangle&  = \langle f,g \rangle \\
& \gtrsim m^{-( r + \frac{1}{\theta}+\frac{1}{2})} \log^\ast(m)^{(d-1)r} \sum_{k \in \, \hyp_n \! \setminus K} 1\\
& =  m^{-( r + \frac{1}{\theta}+\frac{1}{2})} \log^\ast(m)^{(d-1)r} \big( \# \hyp_n - \,\# K \big) \\
& \geq  m^{-( r + {\frac{1}{\theta}-\frac{1}{2}})} \log^\ast(m)^{(d-1)r}.
\end{split}
\end{align}

Using the Cauchy-Schwarz inequality and \eqref{hup}, this implies 
 \label{csinq}
  \begin{equation} m^{-( r + {\frac{1}{\theta}-\frac{1}{2}})} \log^\ast(m)^{(d-1)r} \lesssim \langle f-h,g \rangle \leq \|f - h\|_{L_2}\, \|g\|_{L_2} \lesssim \|f - h\|_{L_2}.
\end{equation}

Since \(K\) and \(a_\bk\) are arbitrary (so that \(h\) can realize any element of \(\Sigma_m\)) this implies, for all \(q \geq 2\),

 \begin{equation} \label{lowend}
\sigma_{m} \left(S^{r}_{\theta}\mathcal{A} \right)_{L_q} \gtrsim  \inf_{h \in \Sigma_{C_n/2}} \|f - h\|_{ {L_2}} \gtrsim m^{-( r + {\frac{1}{\theta}-\frac{1}{2}})} \log^\ast(m)^{(d-1)r}.
\end{equation}
\end{proof}

Results obtained in \cite{Kri23} as well as Theorem \ref{samp_A1} need the best \(m\)-term approximation measured, not in \(L_\infty\), but in \(\ca\)  {instead}.  Therefore we will now also provide a result of this type for weighted Wiener spaces.
Before we prove any involved bounds in this setting we start with a simple auxiliary lemma regarding the best \(m\)-term approximation width between unweighted Wiener  {spaces} of different orders.
\begin{lem}\label{A2A_simple}
For \(r = 0\) and  \(0 < \theta < \eta \leq \infty\) it holds
\be
\sigma_m(\ca_\theta)_{\ca_\eta} \asymp m^{1/\eta-1/\theta} .
\ee
\end{lem}
\begin{proof}
    Applying  {the Stechkin} Lemma \ref{stechkin} to the definition of the norm immediately yields the upper bound, the lower bound is a direct consequence of the proof of Theorem \ref{A2A}.
\end{proof}

We can get best \(m\)-term rates between Wiener spaces by arguing as in \cite[Theorem 3.1]{NNS22}. For  {$\theta=2$,  $\eta\in\{1;2\}$ and $\theta=1$,  $\eta\in\{1;2\}$ }  similar  {results were} shown in  {\cite[Theorem 4.2]{NN22} and} \cite[Theorem 4.3]{NN22},  {respectively,} where, in particular, the asymptotic constants were also computed. The following result generalizes these to a much broader range of parameters.

\begin{theorem}\label{A2A}
For \(r > \big(\frac{1}{\eta} - \frac{1}{\theta}\big)_+\) and   {\( \theta, \eta \in (0,\infty]\)} it holds
\be
\sigma_m(S^r_\theta\ca)_{\ca_\eta} \asymp 
m^{\frac{1}{\eta} - \frac{1}{\theta}-r}\log^\ast(m)^{(d-1)r}.
\ee
\end{theorem}

\bproof
For every \( \bk \in \zz^d \) we define the corresponding weight according to the \(S^1_1\ca\)-norm as
\begin{equation} \label{omega_bk}
    \omega_\bk \coloneqq \prod_{j = 1}^d (\vert k_j \vert +1) .
\end{equation}

 {For \(\bk \in \hyp_{n}\), where \(\hyp_n\) is defined as in relation \eqref{square}, clearly it holds that 
\begin{equation} \label{weight_hyp}
    2^{n-d} < \omega_\bk \leq 2^n.
\end{equation} 
This also implies for any \(n \geq d \)
\begin{equation} \label{order_weights}
    \sup_{\bk \in \hyp_{n-d}} \omega_{\bk} \leq \inf_{\bk \in \hyp_{n}} \omega_{\bk}\leq \sup_{\bk \in \hyp_{n}} \omega_{\bk} \leq \inf_{\bk \in \hyp_{n+d}} \omega_{\bk}.
\end{equation}}

Let \(J:\Z \to \N\) be an ordering, such that the weights \(\omega_{J^{-1}(k)}\) are non-decreasing. We can now extend the definition of \(\omega\) to \(k \in \N\):

\be \label{omega_N}
\omega_k \coloneqq \prod_{j = 1}^d (\vert J^{-1}(k)_j \vert +1) .
\ee
We start with the simpler regime where  {\(\infty \geq \eta \geq \theta > 0\)}. By the submultiplicativity of the best \(m\)-term approximation {widths}, see \eqref{mul_snum} and Lemma \ref{A2A_simple} we get
\begin{align}
\sigma_{2m}(S^r_\theta\ca)_{\ca_\eta} & \leq \sigma_m(\ca_\theta)_{\ca_\eta}\,\, \sigma_m(S^r_\theta\ca)_{\ca_\theta} \\
    & \lesssim m^{1/\eta -1/\theta} \sup_{\Vert f\Vert_{S^r_\theta\ca} \leq 1} \Big(  \sum_{k = m+1}^\infty \big\vert \hat{f}(J^{-1}(k)) \big\vert^\theta\Big)^{1/\theta}\\
  & \leq m^{1/\eta -1/\theta} \omega_{m}^{-r} \sup_{\Vert f\Vert_{S^r_\theta\ca} \leq 1} \Big( \sum_{k = m+1}^\infty \big\vert \hat{f}(J^{-1}(k)) \big\vert^\theta \omega_k^{r\theta} \Big)^{1/\theta}\\
  & \leq m^{1/\eta -1/\theta} \omega_m^{-r}.
\end{align}
Now we just need to analyze the growth of \(\omega_m\). To this end we assume \(m = 2^n\) and consider the step hyperbolic layers from \eqref{square}. By \eqref{C_j} we know that the \(n\)-th  {hyperbolic layer has asymptotically} \( 2^n n^{d-1}\) entries. Taking the sum up to \(n\) yields
\be
 {\sum_{l = 0}^n \# \hyp_l \asymp} \sum_{l = 0}^n 2^l l^{d-1}  \asymp 2^n n^{d-1}.
\ee
Therefore  {by \eqref{weight_hyp} and \eqref{order_weights} we have,
\begin{align}
  2^n \lesssim 2^{n-d} \leq \sup_{\bk \in \hyp_{n-d}} \omega_{\bk} \leq \omega_{2^n n^{d-1}} \leq \inf_{\bk \in \hyp_{n+d}} \omega_\bk\lesssim 2^n,
\end{align} }
i.e. \(2^n \asymp\omega_{2^n n^{d-1}} \). From this we can conclude
\begin{equation} \label{weights_upper}
  \omega_{2^n}  \geq \omega_{\lfloor2^{ n-(d-1)\log^\ast(n) } ( n-(d-1)\log^\ast(n))^{d-1}\rfloor} \asymp  2^{n-(d-1)\log^\ast(n)} = 2^n {n}^{1-d},
\end{equation}
and further for \(n \geq 2(d-1) \log^\ast(n)\)
\begin{equation} \label{weights_lower}
  \omega_{2^n} \leq \omega_{\lceil 2 2^{ n-(d-1)\log^\ast(n) } (n-(d-1)\log^\ast(n))^{d-1}\rceil} \asymp  2^{n-(d-1)\log^\ast(n)} = 2^n {n}^{1-d}.
\end{equation}
This implies 
\be
 \omega_{m} = \omega_{2^n} \asymp 2^{n-(d-1)\log^\ast(n)} = m \log^\ast(m)^{1-d},
\ee
which concludes the first part of the proof after arguing again with monotonicity and employing the embedding from the beginning of the proof.

Now we continue with the more involved situation where  {\(\infty \geq \theta > \eta > 0\)}, here we have the additional condition \(r > \frac{1}{\eta} - \frac{1}{\theta}\).
We set \(m= 2^n\) and choose \(L= \lceil n - (d-1) \log^\ast(n) - \kappa \rceil\) slightly different as before in \eqref{L}, where \(\kappa \in \N\) is chosen later. Note that we always have \(L \leq n\) We start the estimation by decomposing onto our uniform hyperbolic layers and reformulating the \(\ca_\eta\)-norm in terms of the \({S_\theta^r\ca}\)-norm. This gives
\begin{align}    
\sigma_m\big(S_\theta^{r}\ca\big)_{\ca_\eta} & = \sup_{\Vert f \Vert_{S_\theta^r\ca} \leq 1} \inf_{s \in \Sigma_m} \Vert f - s \Vert_{\ca_\eta}\\
    & \leq \sup_{\Vert f \Vert_{S_\theta^r \ca} \leq 1} \inf_{s \in \Sigma_m} \Big(\sum_{\bk \in \Z} \vert \hat{f}(\bk) - \hat{s}(\bk) \vert^\eta \Big)^{1/\eta} \\
    & = \sup_{\Vert f \Vert_{S_\theta^r\ca} \leq 1} \inf_{s \in \Sigma_m} \Big(\sum_{j \in \N_0} \sum_{\bk \in \hyp_j} \vert \hat{f}(\bk) - \hat{s}(\bk) \vert^\eta \Big)^{1/\eta} \\
& { \lesssim \sup_{\Vert f \Vert_{S_\theta^r\ca} \leq 1} \Big(\sum_{j = L+1}^\infty \big( \sum_{\bk \in \hyp_j} \vert \hat{f}(\bk) \vert^\eta \big)^{\eta/\eta } \Big)^{1/\eta}.}
\end{align}

 Where in the last line we need that
 \begin{equation}
     \sum_{j=0}^L C_j \leq C L^{d-1}  \sum_{j=0}^L 2^j \leq C L^{d-1} 2^{L+1} \leq 4 C n^{d-1} 2^{n - \kappa} = \frac{4C}{2^{\kappa}} 2^n \leq 2^n = m.
 \end{equation}
 This is true provided \(\kappa\) is chosen such that \(4C \leq 2^\kappa\).
 Now we estimate the \(\ell_\eta\)-norm of a sequence with at most \(C_j\) non-zero entries (see \eqref{C_j})  {from above by Lemma \ref{CSI},}

     { \begin{align}
    \begin{split} \sigma_m\big(S_\theta^{r}\ca\big)_{\ca_\eta} & \lesssim \sup_{\Vert f \Vert_{S_\theta^r\ca} \leq 1} \Big(\sum_{j = L+1}^\infty C_j^{\frac{\theta- \eta}{\theta}} \big(  \sum_{\bk \in \hyp_j} \vert \hat{f}(\bk) \vert^\theta \big)^{\frac{\eta}{\theta}}\Big)^{\frac{1}{\eta}} \\
   & \asymp \sup_{\Vert f \Vert_{S_\theta^r\ca} \leq 1} \Big(\sum_{j = L+1}^\infty C_j^{\frac{\theta- \eta}{\theta}}  \big( 2^{-j\theta r} \sum_{\bk \in \hyp_j} \omega_{\bk}^{\theta r} \vert \hat{f}(\bk) \vert^\theta \big)^{\frac{\eta}{\theta}}\Big)^{\frac{1}{\eta}} \\
 & = \Big(\sum_{j = L+1}^\infty C_j^{\frac{\theta-\eta}{\theta} }  2^{-r j \eta} \big( \sup_{\Vert f \Vert_{S_\theta^r\ca} \leq 1}  \sum_{\bk \in \hyp_j} \omega_{\bk}^{\theta r} \vert \hat{f}(\bk) \vert^\theta\big)^{\frac{\eta}{\theta}} \Big)^{\frac{1}{\eta}} \\
 & \lesssim \Big(\sum_{j = L+1}^\infty (2^j j^{d-1})^{1 -\frac{\eta}{\theta} } 2^{-rj\eta}\Big)^{\frac{1}{\eta}} \\
 & \leq \Big(\sum_{j = L+1}^\infty 2^{-j( r\eta -1  + \frac{\eta}{\theta} )} j^{(1-\frac{\eta}{\theta} )(d-1)} \Big)^{\frac{1}{\eta}}.
 \end{split}
    \end{align}}
    Since  {this} sum is  {decaying} geometrically we can  {apply Lemma \ref{geo_sum}} (with \(\alpha = 0, \beta = (1-\frac{\eta}{\theta} )(d-1)\) and \(\gamma = r\eta -1  + \frac{\eta}{\theta} \)),
    \begin{align}   
    \begin{split}
    \sigma_m\big({S_\theta^r\ca}\big)_{\ca_\eta}  &   {\lesssim \big(
 2^{-L( r \eta - 1 + \frac{\eta}{\theta} )} L^{(1-\frac{\eta}{\theta} )(d-1)} \big)^{1/\eta} } \\
 & \lesssim  {  2^{-(n-(d-1) \log^\ast (n)-\kappa )(r-1/\eta + 1/\theta)}} n^{(1/\eta-1/\theta)(d-1)} \\
 & \lesssim  m^{-(r-1/\eta+1/\theta)} \log^\ast(m)^{(r-1/\eta+1/\theta)(d-1)} \log^\ast(m)^{(1/\eta-1/\theta)(d-1)} \\
 & = m^{-(r-1/\eta+1/\theta)} \log^\ast(m)^{r(d-1)}.
     \end{split}
    \end{align}
One can again extend this from \(m = 2^n\) to general \(m \in \N\) by monotonicity.

To get a lower bound consider the trigonometric polynomial 
\be
t(\bx) = \sum_{k =1}^{2m} \frac{\omega_{k}^{-r}}{(2m)^{1/\theta} } \exp(2 \pi \mathrm{i} J^{-1}(k) \bx).
\ee
This has the following \(S^r_{\theta}\ca\)-norm
\be
\Vert t \Vert_{S^r_{\theta}\ca} =\Big(\sum_{\bk\in\Z} \prod_{i=1}^d (1+|k_i|)^{r\theta}|\hat{t}(\bk)|^\theta\Big)^{1/\theta} = \Big(\sum_{k=1}^{2m} \omega_k^{r\theta} \frac{\omega_{k}^{-r\theta}}{2m}\Big)^{1/\theta} = 1,
\ee
and therefore lies on the unit sphere of this space. Now it holds
\begin{align} \sigma_m(S^r_{\theta}\ca)_{\ca_\eta} & \geq \inf_{s \in \Sigma_m} \Big(\sum_{\bk \in \Z} \Big\vert \hat{t}(\bk) - \hat{s}(\bk) \Big\vert^\eta \Big)^{1/\eta} \\
  & \geq \inf_{s \in \Sigma_m} \Big(\sum_{k = 1}^{2m} \Big\vert \frac{\omega_{k}^{-r}}{(2m)^{1/\theta} } - \hat{s}(J^{-1}(k)) \Big\vert^\eta \Big)^{1/\eta} \\
  & \geq m^{1/\eta} \frac{\omega_{2m}^{-r}}{(2m)^{1/\theta}} \\
  & \gtrsim m^{1/\eta-1/\theta-r} \log^\ast(m)^{(d-1)r}.
\end{align}
\eproof

Similar results could also be shown by combining results from \cite{Stepanets_UMZh_2001_N8}, 
\cite[Chapter 11]{Stepanets_MAT_2005} 
and \cite{Gao_JAT_2010} with the considerations from \eqref{weights_upper} and \eqref{weights_lower}.

\begin{lem}\label{E_M}
    For \(0 < \theta \leq \infty\) and \( r > \big(1 - \frac{1}{\theta}\big)_+\) it holds
    \begin{equation}
        E_{[-M,M]^d}(S^r_\theta \ca)_{L_\infty} \lesssim 
        \begin{cases}
            M^{-r}, & \theta \leq 1 \\
            M^{1- \frac{1}{\theta}-r}, & \theta > 1.
        \end{cases}
    \end{equation}
\end{lem}
\begin{proof}
    Set \(\widetilde{\theta} = \max\{1, \theta \}\), then it holds
    \begin{align}
        \begin{split}
             E_{[-M,M]^d}(S^r_\theta \ca)_{L_\infty} & \leq  E_{[-M,M]^d}(S^r_\theta \ca)_{\ca} \\
             & = \sup_{\Vert f\Vert_{S^r_\theta \ca} \leq 1} \sum_{\bk \in \Z \setminus [-M,M]^d} \omega^{-r}_\bk \omega^{r}_\bk \vert \hat{f}(\bk)\vert \\
             & \leq \sup_{\Vert f\Vert_{S^r_\theta \ca} \leq 1} \Vert (\omega^{-r}_\bk)_{\bk \in \Z \setminus [-M,M]^d} \Vert_{\ell_{\theta'}}  \Vert (\omega^{r}_\bk \vert \hat{f}(\bk) \vert)_{\bk \in \Z} \Vert_{\ell_{\theta}} \\
             & \leq \Vert (\omega^{-r}_\bk)_{\bk \in \Z \setminus [-M,M]^d} \Vert_{\ell_{\widetilde{\theta}'}}.
        \end{split}
    \end{align}
    For \( \theta \leq 1 \) we have \( {\widetilde{\theta}'} = \infty \) and therefore \( E_{[-M,M]^d}(S^r_\theta \ca)_{L_\infty} \leq M^{-r} \). Otherwise we have \(\widetilde{\theta} = \theta\) and can continue by noting that \(r \theta' > 1\) with
    \begin{align}
        \begin{split}
            E_{[-M,M]^d}(S^r_\theta \ca)_{L_\infty}^{\theta'} & \leq \sum_{\bk \in \Z \setminus [-M,M]^d} \omega^{-r\theta'}_\bk \\
            & \leq d \sum_{\bk \in \zz^{d-1}} \sum_{l \in \zz \setminus [-M,M]} (1+ \vert l\vert)^{-r\theta'} \prod_{j=1}^{d-1} (1+\vert k_j \vert)^{-r\theta'} \\
            & \leq 2d \sum_{l = M +1}^\infty (1+ \vert l\vert)^{-r\theta'} \sum_{\bk \in \zz^{d-1}}  \prod_{j=1}^{d-1} (1+\vert k_j \vert)^{-r\theta'} \\
            &\lesssim \sum_{l = M +1}^\infty (1+ \vert l\vert)^{-r\theta'} \\
            & \lesssim M^{1-r\theta'}.
        \end{split}
    \end{align}
    Where in the second to last line we used the fact that \( \sum_{\bk \in \zz^{d}}  \prod_{j=1}^{d} (1+\vert k_j \vert)^{-r\theta'} \leq C^d\) for some C (this can easily be checked by induction over \(d\)). The inequality in the last line holds either by comparing to the integral or employing the Cauchy condensation test. Taking the \(\theta'\)-th root now gives the assertion.
\end{proof}

\section{Embeddings between Besov-Sobolev and weighted Wiener spaces}

\begin{defi}[Besov space with mixed smoothness]
For \(1 < p < \infty\), \(0 < \theta \leq \infty\) and $r>0$ we define the (periodic) Besov space \(S^{r}_{p,\theta}B\) with mixed smoothness, see also, \cite[Chapter 3]{DTU18} or \cite[Chapter 2]{SchTr87} as
\be \label{def_Besov}
 S^{r}_{p,\theta}B(\tor^d) \coloneqq  {S^{r}_{p,\theta}B \coloneqq} \Big\{ f \in L_p(\tor^d)~:~ \Vert f \Vert_{S^{r}_{p,\theta}B} < \infty \Big\}
\ee
with the (quasi-)norm
\be \label{def_Besov_norm}
\Vert f  \Vert_{S^{r}_{p,\theta}B } \coloneqq \bigg( \sum_{\bj \in \nd }  2^{\Vert \bj \Vert_{\ell_1} r \theta} \Big\Vert \sum_{\bk\in I_{\bj} } \hat{f}(\bk)\exp(2\pi \mathrm{i} \bk \cdot)\Big\Vert_{ {L_p}}^\theta\bigg)^{\frac{1}{\theta}},
\ee 
(with the usual modification if $\theta=\infty$), where \(I_ {\mathbf{\bj}} \) is defined in \eqref{Ik}.
\end{defi}

The space $S^r_pW$ of functions with bounded mixed derivative  {is} defined in the following way.

 \begin{defi}[Sobolev space with bounded mixed derivative] \label{def2Sob_L-P}
For \(1 < p < \infty\) and $r\geq 0$ we define the (periodic) Sobolev space \(S^{r}_{p}W\) with mixed smoothness (derivative)
\be \label{def_Sobolev_L-P}
  S^r_p W(\mathbb{T}^d)
  := S^{r}_{p}W 
 \coloneqq 
 \Big\{ f \in L_p(\tor^d)~:~ \Vert f \Vert_{S^{r}_{p}W} < \infty \Big\}
\ee
with the norm
\be \label{def_Sobolev_norm_L-P}
\Vert f  \Vert_{S^{r}_{p}W } \coloneqq \bigg\Vert \bigg( \sum_{\bj \in \nd } \Big( 2^{\Vert \bj \Vert_{\ell_1} r} \Big\vert \sum_{\bk\in I_{\bj} } \hat{f}(\bk)\exp(2\pi \mathrm{i} \bk \cdot) \Big\vert \Big)^2 \bigg)^{\frac{1}{2}} \bigg\Vert_{ {L_p}} .
\ee 
 In the case $r=0$ we have $S^{0}_{p}W:=L_p$ and put $\Vert f  \Vert_{S^{0}_{p}W }:=\Vert f \Vert_{L_p}$. 
  \end{defi}

\begin{theorem}\label{SB_subset_SA_theta<2<p}
If $0<\theta \leq 2\leq p<\infty$, $r\geq 0$  then 
 \be  \label{SB_subset_SA_theta<2<p_EMB}
 S^{r+\frac{1}{\theta}-\frac{1}{2}}_{p,\theta}B
 \hookrightarrow 
 S^{r}_{\theta}\mathcal{A}\,,
 \ee
 where the embedding has norm $1$. 
\end{theorem}  
\begin{proof} 
Using the H\"{o}lder inequality for $1\leq  { \, \frac{2}{\theta}=:\tau } <\infty$,  {\eqref{Vol_Ik}} and $\|\cdot\|_{ {L_2}}\leq \|\cdot\|_{ {L_p}}$, $2\leq p$, we have 
\begin{align} \label{SB_subset_SA_theta_leq_2_leq_p}
  \begin{split}
\|f\|_{{S^{r}_{\theta}\mathcal{A}}}
    & = \bigg(\sum\limits_{\bj \in \mathbb{N}_0^d} \sum\limits_{\bk \in I_{\bj}}
    \prod_{i=1}^d (1+|k_i|)^{r \theta}
    |\hat{f}(\bk)|^{\theta} \bigg)^{\frac{1}{\theta}}  \\
    & \ \leq \bigg( \sum\limits_{\bj \in \mathbb{N}_0^d} \Big(\sum\limits_{\bk \in I_{\bj}}
    \big(\prod_{i=1}^d (1+|k_i|)^{r \theta} \big)^{\frac{2}{2-\theta}} \Big)^{1-\frac{\theta}{2}}
    \Big(\sum\limits_{\bk \in I_{\bj}} |\hat{f}(\bk)|^2\Big)^{\frac{\theta}{2}} \bigg)^{\frac{1}{\theta}} \\
    & \ \leq  \bigg( \sum\limits_{\bj \in \mathbb{N}_0^d} 2^{ {\Vert \bj \Vert_{\ell_1}} r \theta} 
    \big(\#I_{\bj}\big)^{1-\frac{\theta}{2}}
    \Big\Vert \sum_{\bk\in I_{\bj} } \hat{f}(\bk)\exp(2\pi \mathrm{i} \bk \cdot)\Big\Vert_{ {L_2}}^\theta \bigg)^{\frac{1}{\theta}}  \\
    & \ \leq  \bigg( \sum\limits_{\bj \in \mathbb{N}_0^d} 2^{ {\Vert \bj \Vert_{\ell_1}} (r+\frac{1}{\theta}-\frac{1}{2}) \theta} \Big\Vert \sum_{\bk\in I_{\bj} } \hat{f}(\bk)\exp(2\pi \mathrm{i} \bk \cdot)\Big\Vert_{ {L_p}}^\theta \bigg)^{\frac{1}{\theta}}  \\
    & \  = \big\|f\big\|_{S^{r+\frac{1}{\theta}-\frac{1}{2}}_{p,\theta}B}\,.
  \end{split}
\end{align}
  {This demonstrates the embedding has norm at most $1$. By checking the fooling function  $f\equiv 1$, we also get that this embedding has norm at least $1$. }
\end{proof}

\begin{rem}\label{Besov_subset_Wiener_theta=1_p=2_REM}
The case $\theta=1$, $p=2$, $r=0$ of Theorem \ref{SB_subset_SA_theta<2<p} in the $1$-dimensional setting (as a particular case of $d$-dimensional setting for isotropic smoothness)  can be found in \cite[Lemma 3.3]{CoKuSi_JMAA_19}.
\end{rem}

\begin{theorem}\label{SB+1/p_subset_SA}
 If 
  {$1<p\leq 2$, $0<\theta\leq 2$,}
 $r\geq 0$ then
 \be \label{Besov_subset_Wiener_0<p,theta<2}  
 S^{r+{\frac{1}{p}}+{\frac{1}{\theta}-1}}_{p,\theta}B
 \hookrightarrow S^{r}_{\theta}\mathcal{A} \,.
 \ee  
\end{theorem}

\begin{proof}  Let us rewrite 
 \eqref{Besov_subset_Wiener_0<p,theta<2} 
 in the form:
 \be \label{Besov_subset_Wiener_0<p,theta<2_DETAILED} 
 S^{r+{\frac{1}{p}-\frac{1}{2}}+{\frac{1}{\theta}-\frac{1}{2}}}_{p,\theta}B
 \hookrightarrow S^{r+{\frac{1}{\theta}-\frac{1}{2}}}_{2,\theta}B
 \hookrightarrow S^{r}_{\theta}\mathcal{A} .
 \ee  
 
The first embedding in \eqref{Besov_subset_Wiener_0<p,theta<2_DETAILED} 
 {holds by} 
\cite[Lemma 3.4.1 (i)]{DTU18}. 

The proof of the second embedding in \eqref{Besov_subset_Wiener_0<p,theta<2_DETAILED}  was done in Theorem \ref{SB_subset_SA_theta<2<p} for $p=2$.

\end{proof} 

 \begin{theorem}\label{SW+2/p_subset_SA_p}
 If $1<p\leq 2$, $r\geq 0$ then 
 \be  \label{W_subset_Wiener_1<p<2}
  S^{r+\frac{2}{p}-1}_{p}W
  \hookrightarrow S^{r}_{p}\mathcal{A} \,.
 \ee  
  \end{theorem} 
\begin{proof} The case $p=2$ of the embedding \eqref{W_subset_Wiener_1<p<2} is obvious. Let us rewrite 
 \eqref{W_subset_Wiener_1<p<2}  {for $1<p<2$} 
 in the form:
 \be  \label{W_subset_Wiener_1<p<2_DETAILED}
 S^{r+2(\frac{1}{p}-\frac{1}{2})}_{p}W
 \hookrightarrow S^{r+{\frac{1}{p}-\frac{1}{2}}}_{2,p}B
 \hookrightarrow S^{r}_{p}\mathcal{A} \,.
 \ee 
 
 The first embedding in \eqref{W_subset_Wiener_1<p<2_DETAILED}  {holds by}  
 \cite[Lemma 3.4.3]{DTU18}. 

 The proof of the second embedding in \eqref{W_subset_Wiener_1<p<2_DETAILED}  was done in Theorem \ref{SB+1/p_subset_SA}.
 \end{proof}

\begin{theorem}\label{SA^-1/p_subset_SB_2<p,theta<infty}
 If 
  {$2\leq p<\infty$, $2\leq \theta\leq\infty$,}
 $r>0$ then
 \be  \label{SA_subset_SB_2<p,theta<infty}
  S^{r+1-\frac{1}{p}-\frac{1}{\theta}}_{\theta}\mathcal{A}
  \hookrightarrow S^r_{p,\theta}B .
 \ee 
\end{theorem}  

\begin{proof}
  Let us rewrite 
 \eqref{SA_subset_SB_2<p,theta<infty}
 in the form:
 \be  \label{SA_subset_SB_2<p,theta<infty_DETAILED}
 S^{r-(\frac{1}{p}-\frac{1}{2})-(\frac{1}{\theta}-\frac{1}{2})}_{\theta}\mathcal{A}
 \hookrightarrow S^{r-(\frac{1}{p}-\frac{1}{2})}_{2,\theta}B
 \hookrightarrow S^r_{p,\theta}B .
 \ee 

 { Let $\sigma:=r-(\frac{1}{p}-\frac{1}{2})$. Then the } proof of the first embedding in \eqref{SA_subset_SB_2<p,theta<infty_DETAILED} 
 is based on using the Hölder inequality for  {$1\leq   \frac{\theta}{2}=:\tau \leq\infty$} 
  {and \eqref{Vol_Ik}}:  
 \begin{align}
 \begin{split}
\Vert f  \Vert_{S^{ \, {\sigma} }_{2,\theta}B } 
    & =
     {  
     \bigg( \sum\limits_{\bj \in \mathbb{N}_0^d} 2^{ {\Vert \bj \Vert_{\ell_1}}   {\sigma}  \theta} 
    \Big\Vert \sum_{\bk\in I_{\bj} } \hat{f}(\bk)\exp(2\pi \mathrm{i} \bk \cdot)\Big\Vert_{L_2}^\theta \bigg)^{\frac{1}{\theta}} 
    }
    \\
    & = \bigg( \sum_{ {\bj} \in \nd } 2^{ {\Vert \bj \Vert_{\ell_1}}  {\sigma}  \theta} \big(\sum_{\bk\in I_{\bj} }  { 1 \cdot } |\hat{f}(\bk)|^2  \big)^\frac{\theta}{2} \bigg)^\frac{1}{\theta} \\
& \leq \bigg( \sum_{ {\bj} \in \nd } 2^{  {\Vert \bj \Vert_{\ell_1}}  {\sigma}  \theta } \bigg( 
  {
\Big(\sum_{\bk\in I_{\bj} } 1^\frac{\theta}{\theta-2} \Big)^\frac{\theta-2}{\theta}
        }
  {        
\Big(\sum_{\bk\in I_{\bj}  } \big(|\hat{f}(\bk)|^2\big)^\frac{\theta}{2} \Big)^\frac{2}{\theta}  
        }
\bigg)^\frac{\theta}{2} \bigg)^\frac{1}{\theta} \\
&  {=} \, \bigg( \sum_{ {\bj} \in \nd } 2^{ {\Vert \bj \Vert_{\ell_1}}   {\sigma} \theta} 
  {
\big(\#I_{\bj}\big)^{\frac{\theta}{2}-1} 
        }
\sum_{\bk\in I_{\bj} } |\hat{f}(\bk)|^{\theta}   \bigg)^\frac{1}{\theta} \\
& \asymp \bigg( \sum_{\bk\in\Z} \prod_{i=1}^d (1+|k_i|)^{ ( \, {\sigma} -(\frac{1}{\theta}-\frac{1}{2})) \theta}|\hat{f}(\bk)|^\theta \bigg)^\frac{1}{\theta} \\
& = \big\Vert f  \big\Vert_{S^{ \, {\sigma} -(\frac{1}{\theta}-\frac{1}{2})}_{\theta}\ca} .
 \end{split}
 \end{align}

 The second embedding in \eqref{SA_subset_SB_2<p,theta<infty_DETAILED}  {holds by}  
 \cite[Lemma 3.4.1 (i)]{DTU18}. 
 \end{proof}

 \begin{theorem}\label{SA-2/p_subset_SW_p}
 If $2\leq p<\infty$, $r\geq 0$ then
 \be  \label{Wiener_subset_W_2<p<infty}
 S^{r+1-\frac{2}{p}}_{p}\mathcal{A}
 \hookrightarrow S^r_{p}W \,.
 \ee
 \end{theorem} 

\begin{proof}
 {The case $p=2$ of the embedding \eqref{Wiener_subset_W_2<p<infty} is obvious.}
 Let us rewrite \eqref{Wiener_subset_W_2<p<infty}  {for $2<p<\infty$}
 in the form:
 \be  \label{Wiener_subset_W_2<p<infty_DETAILED}
 S^{r+2(\frac{1}{2}-\frac{1}{p})}_{p}\mathcal{A}
 \hookrightarrow S^{r+(\frac{1}{2}-\frac{1}{p})}_{2,p}B
 \hookrightarrow S^r_{p}W \, .
 \ee

 The  first embedding 
 in \eqref{Wiener_subset_W_2<p<infty_DETAILED} was proven in 
 Theorem \ref{SA^-1/p_subset_SB_2<p,theta<infty}.

 The second embedding 
 in \eqref{Wiener_subset_W_2<p<infty_DETAILED}  {holds by}  \cite[Lemma 3.4.2]{DTU18}. 
\end{proof}

\section{Asymptotic bounds for sampling recovery of  functions}

\subsection{Weighted Wiener spaces} \label{asymp_W}
As mentioned in the introduction, the Kolmogorov widths of mixed Wiener spaces have already been studied in \cite{NNS22}, where it was shown that they behave like
\begin{equation}
d_m\big( S^{r}_{1}\mathcal{A}\big)_{L_2} \asymp m^{- r} \log^\ast(m)^{(d-1) r }.
\end{equation}
Kolmogorov/linear widths form a lower bound on the linear sampling widths if the error is measured in $L_2$. This is a trivial consequence of the Definition \ref{defi_l_samp}. Hence, we also know
\begin{equation} \label{lsn_low}
\varrho^{\operatorname{lin}}_m\big(S^{r}_{1}\mathcal{A}\big)_{L_2} \gtrsim m^{- r} \log^\ast(m)^{(d-1) r}.
\end{equation}
The Gelfand widths of weighted Wiener spaces were studied in \cite[Theorem 3.2, Corollary 3.3]{Mo23}    {for \(0 < \theta \leq 2\) and \(r >  \big(1-\frac{1}{\theta}\big)_+\)}, they behave like 
\begin{equation} \label{nlsn_low}
c_m\big( S^{r}_{\theta}\mathcal{A}\big)_{L_2} \asymp m^{- (r+{\frac{1}{\theta}}-\frac{1}{2})} \log^\ast(m)^{(d-1) r }.
\end{equation}

The Gelfand widths \(c_m\), similarly to the Kolmogorov widths \(d_m\), give a lower bound, but on the non-linear sampling widths, see e.g.\ \cite[Lemma B.1]{JUV23}. We therefore get
\begin{equation}
\varrho_m\big(S^{r}_{\theta}\mathcal{A}\big)_{L_2} \gtrsim m^{- (r+{\frac{1}{\theta}}-\frac{1}{2})} \log^\ast(m)^{(d-1) r }.
\end{equation}

Using Proposition \ref{samp_Lp} we can get an upper bound on the non-linear sampling widths as well.

\begin{theorem}\label{nlsn_up}
For \(0 < \theta \leq \infty\), \(2 \leq q \leq \infty\) and \( r > \big(1 - \frac{1}{\theta}\big)_+\) it holds

\begin{equation}
 {\varrho_m\big(S^{r}_{\theta}\ca\big)_{L_q} \lesssim m^{-( r+{\frac{1}{\theta}}+\frac{1}{q}-1)} \log^\ast(m)^{(d-1) r + 3 ( r + {\frac{1}{\theta}+\frac{1}{q}-1})+ \frac{1}{2}}.}
\end{equation}
In particular we get for \(\theta=1\) and \(q=2\),
\begin{equation}
\varrho_m\big(S^{r}_{1}\ca\big)_{L_2} \lesssim m^{-( r+\frac{1}{2})} \log^\ast(m)^{(d-1) r +  {3r+2} } .
\end{equation}
\end{theorem}

\begin{proof}
Choose \(M= \big\lceil n^{(r+{\frac{1}{\theta}-\frac{1}{2}})\gamma^{-1}}\big\rceil\) and \(n \in \N\) maximal according to \eqref{m_vs_n} i.e. 
\begin{equation} \label{n}
n = \Big\lfloor C' m \log^\ast(m)^{-3} (r + \frac{1}{\theta} - \frac{1}{2})^{-1}\gamma  \Big\rfloor,
\end{equation}
where  \(C'\) depends only on \(C\) from \eqref{m_vs_n}, \(\theta\) and \(r\) and \(\gamma = r\) for \(\theta\leq 1\) and \(\gamma = r + \frac{1}{\theta} - 1\) otherwise.
Then we get from Proposition \ref{samp_Lp},
\begin{align}
\begin{split}
    \varrho_m\big(S^{r}_{\theta}\ca\big)_{L_q} & \leq C n^{1/2-1/q}\left( \sigma_n\big( S^{r}_{\theta}\ca\big)_{L_\infty} + E_{[-M,M]^d}\big(S^{r}_{\theta}\ca\big)_{L_\infty}\right)\\
    & \lesssim C n^{1/2-1/q} (n^{-(r+{\frac{1}{\theta}-\frac{1}{2}})} \log^\ast(n)^{(d-1)r+\frac{1}{2}} + M^{-\gamma}) \\
    & \leq  C n^{-(r+{\frac{1}{\theta}+\frac{1}{q}-1})} \log^\ast(n)^{(d-1)r+\frac{1}{2}} \\
    & \leq  C m^{-(r+{\frac{1}{\theta}+\frac{1}{q}-1})} \log^\ast(m)^{(d-1)r+\frac{1}{2} + 3(r+{\frac{1}{\theta}+\frac{1}{q}-1})} {,}
    \end{split}
\end{align}
where we use Theorem \ref{sigma_m_A_0<p<infty,q_bigger=_2} and Lemma \ref{E_M} in the second line.

\end{proof}

This result generalizes \cite[Corollary 4.4 and Remark 4.5 (i)]{JUV23} in multiple directions. Not only do we consider the more general class of weighted Wiener  {(\(\theta\) different from 1)} spaces but also reduce the condition  {on \(r\) from \(r > 1/2\) to \(r>(1 - 1/\theta)_+\). In the setting where \(\theta = 1\), as in \cite{JUV23} this reduces to \(r > 0\). In addition} we also consider the situation where the error is measured in Lebesgue spaces other than \(L_2\).

We can still improve this result slightly by using Theorem \ref{samp_A1} and Theorem \ref{A2A} instead of Proposition \ref{samp_Lp}. Since we avoid the \(L_\infty\) norm, we can get rid of a factor \(\sqrt{\log^\ast(m)}\) that appears in Theorem \ref{sigma_m_A_0<p<infty,q_bigger=_2}. 

\begin{cor}\label{nlsn_up_good}
For \(0< \theta  {\, \leq \,} \infty\), \(2 \leq q \leq \infty \) and  {\(r > \big(1 - \frac{1}{\theta}\big)_+\)} we get 
\begin{equation} 
 {\varrho_{m}\big(S_\theta^r \ca \big)_{L_q} \leq Cm^{-( r+{\frac{1}{\theta}+\frac{1}{q}-1})} \log^\ast(m)^{(d-1) r + 3 ( r + {\frac{1}{\theta}+\frac{1}{q}-1})}.}
\end{equation}
\end{cor}
 {
\begin{proof}
Let \(M,n\) as in Theorem \ref{nlsn_up} above. Then we get from Theorem \ref{samp_A1} together with Theorem \ref{A2A} (with \(\eta =1\))
    \begin{align}
        \begin{split}           \varrho_m\big(S^{r}_{\theta}\ca\big)_{L_q} & \leq C n^{1/2-1/q}\left(n^{-1/2} \sigma_n\big(S^{r}_{\theta}\big)_{\ca} + E_{[-M,M]^d}\big(S^{r}_{\theta}\ca\big)_{L_\infty}\right) \\
    & \leq C n^{1/2-1/q} (n^{-(r+{\frac{1}{\theta}-1+\frac{1}{2}})} \log^\ast(n)^{(d-1)r} + M^{-\gamma}) \\
    & \leq  C n^{-(r+{\frac{1}{\theta}+\frac{1}{q}-1})} \log^\ast(n)^{(d-1)r}\\
    & \leq  C m^{-(r+{\frac{1}{\theta}+\frac{1}{q}-1})} \log^\ast(m)^{(d-1)r + 3(r+{\frac{1}{\theta}+\frac{1}{q}-1})}.
        \end{split}
    \end{align}
\end{proof}
}

\begin{rem}\label{sqrtlog}
    Now we compare the lower bounds on the linear sampling widths \eqref{lsn_low} and the upper bounds on the non-linear sampling widths from Corollary \ref{nlsn_up_good}. 
    We see that non-linear methods perform better than linear ones by a rate of \(m^{-\frac{1}{2}} \log^\ast(m)^{3r + \frac{3}{2}}\) for recovery of functions from mixed Wiener {spaces}  {(\(\theta =1\))}, with the error measured in \(L_2\). The main rate of this difference is independent of \(r\), and since our result, unlike the previous ones in \cite{JUV23}, holds for all \(r>0\)  {(in case of \(q=2\) and \(\theta = 1\))}, we know for situations with very low smoothness that the error of linear methods decays very slowly, while the error of suitable non-linear algorithms always decays with rate at least  \(1/2\). 
    The additional logarithmic term \(\log^\ast(m)^{3r + \frac{3}{2}}\) does not depend on the dimension \(d\) and is the consequence of the conditions on \(m\) in Theorem \ref{samp_A1}.
\end{rem}

\subsection{Besov-Sobolev spaces}

\begin{theorem}\label{sigma_m_Besov_in_Wiener} 
Let 
 $1<p\leq 2$, $0<\theta\leq 2$.   

 1) For $\theta \leq \eta\leq\infty$ and $r>\frac{1}{p}+\frac{1}{\theta}-1$,
   it holds
\begin{equation}\label{sigma_m(S^r_p,theta_B)_A_q}
   \sigma_{m}(S^{r}_{p,\theta}B)_{\mathcal{A}_\eta} \asymp
   m^{-(r+1 -\frac{1}{p}- \frac{1}{\eta})}
   \log^\ast(m)^{(d-1)(r+1 -\frac{1}{p}- \frac{1}{\theta})}  .
\end{equation}

 2)
 In the additional case where $r=\frac{1}{p}+\frac{1}{\theta}-1$ and  
 $\theta < \eta\leq\infty$,
 it holds
\begin{equation}\label{sigma_m(S^r=_2,theta_B)_{L_2}}
   \sigma_{m}(S^{r}_{p,\theta}B)_{\mathcal{A}_\eta} \asymp
   m^{-(\frac{1}{\theta}- \frac{1}{\eta} )} .
\end{equation}
\end{theorem} 

\begin{proof}
 Applying Theorem \ref{SB+1/p_subset_SA}, that is  {the} embedding  
 $S^{r}_{p,\theta}B
 \hookrightarrow S^{r+1-{\frac{1}{p}}-{\frac{1}{\theta}}}_{\theta}\mathcal{A}$, $r\geq\frac{1}{p}+\frac{1}{\theta}-1$,  and also Lemma \ref{A2A_simple}  {in the situation of 2)} or Theorem \ref{A2A}  {in the regime of 1)}, we derive
 $$
 \sigma_{m}(S^{r}_{p,\theta}B)_{\mathcal{A}_\eta} 
\lesssim 
 \sigma_{m}(S^{r+1-{\frac{1}{p}}-{\frac{1}{\theta}}}_{\theta}\hspace{-1pt}\mathcal{A})_{\mathcal{A}_\eta}
 \asymp
 m^{-(r+1 -\frac{1}{p}- \frac{1}{\eta} )} (\log^\ast \hspace{-1pt} m)^{(d-1)(r+1 -\frac{1}{p}- \frac{1}{\theta})} . 
 $$

 Let us prove the lower bounds.
 For \(m \asymp 2^n n^{d-1}\), to get the corresponding lower bounds in \eqref{sigma_m(S^r_p,theta_B)_A_q} and \eqref{sigma_m(S^r=_2,theta_B)_{L_2}} we  construct a fooling function:
\begin{equation}
f(\bx) = C m^{-( r + 1 - \frac{1}{p})} \log^\ast(m)^{(d-1)(r+ 1 - \frac{1}{p} - \frac{1}{\theta})} \mathcal{D}_{\hyp_n} (\bx)  \,,
\end{equation}
where
\begin{equation}
\mathcal{D}_{\hyp_n} (\bx) = \sum_{\bk \in\hyp_n} \exp(2 \pi \mathrm{i} \bk \bx) \,
\end{equation}
and $\hyp_n$ is defined by \eqref{square}. This function has \(S^{r}_{p,\theta}B\)-norm
\begin{align}  \label{fnormedB}
\begin{split}
\|f\|_{S^{r}_{p,\theta}B} 
& \asymp m^{-( r + 1 - \frac{1}{p})} \log^\ast(m)^{(d-1)( r + 1 - \frac{1}{p} - \frac{1}{\theta})} 
\|\mathcal{D}_{\hyp_n}\|_{S^{r}_{p,\theta}B}\\
& \asymp m^{-( r + 1 - \frac{1}{p})} \log^\ast(m)^{(d-1)( r + 1 - \frac{1}{p} - \frac{1}{\theta})} \\
& \times \Big(\sum_{ {\Vert \bj \Vert_{\ell_1}}=n} 2^{  {\Vert \bj \Vert_{\ell_1}} r \theta } \Big\Vert \sum_{\bk\in I_{\bj} }  {\exp(2 \pi \mathrm{i} \bk \cdot)} \Big\Vert_{ {L_p}}^\theta  \Big)^\frac{1}{\theta}\\
& \asymp m^{-( r + 1 - \frac{1}{p})} \log^\ast(m)^{(d-1)( r + 1 - \frac{1}{p} - \frac{1}{\theta})} 2^{ r n } \Big(\sum_{ {\Vert \bj \Vert_{\ell_1}}=n} 2^{  {\Vert \bj \Vert_{\ell_1}} (1-\frac{1}{p}) \theta } \Big)^\frac{1}{\theta}\\
& \asymp m^{-( 1 - \frac{1}{p})} \log^\ast(m)^{(d-1)( 1 - \frac{1}{p} - \frac{1}{\theta})} 2^{ n ( 1-\frac{1}{p}) } \Big(\sum_{ {\Vert \bj \Vert_{\ell_1}}=n} 1 \Big)^\frac{1}{\theta} \\
& \asymp 1 ,
\end{split}
\end{align}
where  {\(I_{\bj}\) is defined in \eqref{Ik} and} we used  {(see e.g.\ \cite[(2.17)]{DTU18})}
 $$
\Big\Vert \sum_{\bk\in I_{\bj} } \exp(2\pi \mathrm{i} \bk \cdot)\Big\Vert_{ {L_p}}
 \asymp 2^{  {\Vert \bj \Vert_{\ell_1}} (1-\frac{1}{p})} , \ \ \ 
 1<p<\infty , 
 $$
 and
 $$
\sum_{ {\Vert \bj \Vert_{\ell_1}}=n} 1  
\asymp n^{d-1} .
 $$

Due to \eqref{C_j}, our choice \(m \asymp 2^n n^{d-1}\) and $\# \hyp_n \geq 2m$, it holds
 \begin{align} 
\sigma_m(S^{r}_{p,\theta}B)_{\ca_\eta} 
& \gtrsim m^{-( r + 1 - \frac{1}{p})} \log^\ast(m)^{(d-1)(r+ 1 - \frac{1}{p} - \frac{1}{\theta})} \sigma_m(\mathcal{D}_{\hyp_n})_{\ca_\eta} \\
& = m^{-( r + 1 - \frac{1}{p})} \log^\ast(m)^{(d-1)(r+ 1 - \frac{1}{p} - \frac{1}{\theta})} 
\big( \#    \hyp_n - \,   
%
 m \big)^{1/\eta} \\
  &  { {\gtrsim}} \,\,  m^{-(r+1 -\frac{1}{p}- \frac{1}{\eta} )}
   \log^\ast(m)^{(d-1)(r+1 -\frac{1}{p}- \frac{1}{\theta})} .
 \end{align} 
\end{proof}

\begin{theorem} \label{WinA}
Let 
 $1<p\leq 2$.  

 1) For $p \leq \eta\leq\infty$ and $r>\frac{2}{p}-1$,
   it holds
 $$ 
   \sigma_{m}(S^{r}_{p}W)_{\mathcal{A}_\eta} \asymp
   m^{-(r+1 -\frac{1}{p}- \frac{1}{\eta} )}
   \log^\ast(m)^{(d-1)(r+1 -\frac{2}{p})} .
 $$ 

 2) In the additional case where $r=\frac{2}{p}-1$ and 
 $p < \eta\leq\infty$,
   it holds
\begin{equation}\label{sigma_m(S^r=_2,theta_W)_{L_2}}
   \sigma_{m}(S^{r}_{p}W)_{\mathcal{A}_\eta} \asymp
   m^{-(\frac{1}{p}- \frac{1}{\eta} )} .
\end{equation}
\end{theorem} 

\begin{proof}
 The upper bounds are based on Theorem \ref{SW+2/p_subset_SA_p}, that is, the embedding  
 $S^{r}_{p}W
 \hookrightarrow S^{r+1-{\frac{2}{p}}}_{p}\mathcal{A}$, $r\geq\frac{2}{p}-1$,  and Theorem \ref{A2A},  {or in the regime of 2) Lemma \ref{A2A_simple}}:
  $$
 \sigma_{m}(S^{r}_{p}W)_{\mathcal{A}_\eta} 
\lesssim 
 \sigma_{m}(S^{r+1-{\frac{2}{p}}}_{p}\mathcal{A})_{\mathcal{A}_\eta}
 \asymp
 m^{-(r+1 -\frac{1}{p}- \frac{1}{\eta} )}
\log^\ast(m)^{(d-1)(r+1 -\frac{2}{p})} . 
 $$

 The lower bounds are based on the embedding $S^{r}_{p,p}B
 \hookrightarrow S^r_{p}W$,  {$1<p\leq 2$} (see, for instance, \cite[Lemma 3.4.1 (iv)]{DTU18}), and Theorem \ref{sigma_m_Besov_in_Wiener}:
  $$
 \sigma_{m}(S^{r}_{p}W)_{\mathcal{A}_\eta} 
\gtrsim 
 \sigma_{m}(S^{r}_{p,p}B)_{\mathcal{A}_\eta}
 \asymp
 m^{-(r+1 -\frac{1}{p}- \frac{1}{\eta} )}
\log^\ast(m)^{(d-1)(r+1 -\frac{2}{p})} . 
 $$
  \end{proof}

\begin{theorem}\label{sigma_m_Besov_in_Wiener_q<theta} 
Let 
 $1<p\leq 2$,    
$0<\eta \leq \theta\leq 2$ and $r>\frac{1}{p}+\frac{1}{\eta}-1$,
   it holds
\begin{equation}\label{sigma_m(S^r_p,theta_B)_A_q<theta}
   \sigma_{m}(S^{r}_{p,\theta}B)_{\mathcal{A}_\eta} \asymp
   m^{-(r+1 -\frac{1}{p}- \frac{1}{\eta} )}
   \log^\ast(m)^{(d-1)(r+1 -\frac{1}{p}- \frac{1}{\theta})}  .
\end{equation}
\end{theorem} 

\begin{proof}
The upper bounds are based on Theorem \ref{SB+1/p_subset_SA}, that is, the  embedding  
 $S^{r}_{p,\theta}B
 \hookrightarrow S^{r+1-{\frac{1}{p}}-{\frac{1}{\theta}}}_{\theta}\mathcal{A}$, $r\geq\frac{1}{p}+\frac{1}{\theta}-1$,  and Theorem \ref{A2A} with the smoothness restriction $r+1-\frac{1}{p}-\frac{1}{\theta}>\frac{1}{\eta}-\frac{1}{\theta}$, that is $r>\frac{1}{p}+\frac{1}{\eta}-1$:
 $$
 \sigma_{m}(S^{r}_{p,\theta}B)_{\mathcal{A}_\eta} 
\lesssim 
 \sigma_{m}(S^{r+1-{\frac{1}{p}}-{\frac{1}{\theta}}}_{\theta}\hspace{-1pt}\mathcal{A})_{\mathcal{A}_\eta}
 \asymp
 m^{-(r+1 -\frac{1}{p}- \frac{1}{\eta} )}
\log^\ast(m)^{(d-1)(r+1 -\frac{1}{p}- \frac{1}{\theta})} . 
 $$

The proof of the lower bound is the same as in Theorem  \ref{sigma_m_Besov_in_Wiener}.
\end{proof}

\begin{theorem} \label{WinA_q<p}
Let 
 $1<p\leq 2$,  
 $0<\eta\leq p$ and $r>\frac{1}{p}+\frac{1}{\eta}-1$,
   it holds
 $$ 
   \sigma_{m}(S^{r}_{p}W)_{\mathcal{A}_\eta} \asymp
   m^{-(r+1 -\frac{1}{p}- \frac{1}{\eta} )}
   \log^\ast(m)^{(d-1)(r+1 -\frac{2}{p})} .
 $$ 
\end{theorem} 

\begin{proof}
 The upper bounds are based on the Theorems \ref{SW+2/p_subset_SA_p}, that is, the embedding  
 $S^{r}_{p}W
 \hookrightarrow S^{r+1-{\frac{2}{p}}}_{p}\mathcal{A}$, $r\geq\frac{2}{p}-1$,  and Theorem \ref{A2A} with the smoothness restriction $r+1-\frac{2}{p}>\frac{1}{\eta}-\frac{1}{p}$, that is $r>\frac{1}{p}+\frac{1}{\eta}-1$:
  $$
 \sigma_{m}(S^{r}_{p}W)_{\mathcal{A}_\eta} 
\lesssim 
 \sigma_{m}(S^{r+1-{\frac{2}{p}}}_{p}\mathcal{A})_{\mathcal{A}_\eta}
 \asymp
 m^{-(r+1 -\frac{1}{p}- \frac{1}{\eta} )}
\log^\ast(m)^{(d-1)(r+1 -\frac{2}{p})} . 
 $$
 
 The lower bounds are based on the embedding $S^{r}_{p,p}B
 \hookrightarrow S^r_{p}W$ (see, for instance, \cite[Lemma 3.4.1 (iv)]{DTU18}) and Theorem \ref{sigma_m_Besov_in_Wiener_q<theta}:
  $$
 \sigma_{m}(S^{r}_{p}W)_{\mathcal{A}_\eta} 
\gtrsim 
 \sigma_{m}(S^{r}_{p,p}B)_{\mathcal{A}_\eta}
 \asymp
 m^{-(r+1 -\frac{1}{p}- \frac{1}{\eta} )}
\log^\ast(m)^{(d-1)(r+1 -\frac{2}{p})} . 
 $$
\end{proof} 
We may now use these embeddings, or more specifically the regime where we embed into \(\ca\), to again get results for the sampling numbers of these spaces measured in \(L_q\), \(2 \leq q \leq \infty\), by applying Theorem \ref{samp_A1}.

\begin{cor}\label{varrho_for_SB_SW_above}
 For \(  {1<p \leq 2} \leq q \leq \infty\),  {$1\leq \theta \leq 2$}  and \(r > 1/p\) we have
\begin{equation}\label{samp_SB_above}
           \varrho_{m}(S^{r}_{p,\theta}B)_{L_q} \lesssim
   m^{-(r -\frac{1}{p} +\frac{1}{q})}
   \log^\ast(m)^{(d-1)(r+1 -\frac{1}{p}- \frac{1}{\theta}) +3(r -\frac{1}{p} +\frac{1}{q}) },
    \end{equation}
    as well as  {for \(1<p \leq 2 \leq q \leq \infty\) and \(r > 1/p\) we have}
\begin{equation}\label{samp_SW_above}
   \varrho_{m}(S^{r}_{p}W)_{L_q} \lesssim
   m^{-(r-\frac{1}{p}+\frac{1}{q})}
   \log^\ast(m)^{(d-1)(r+1 -\frac{2}{p}) + 3(r-\frac{1}{p}+\frac{1}{q})} .
\end{equation}
\end{cor}

\begin{proof}
Here, to prove \eqref{samp_SB_above}, we follow the idea of the proof of Corollary~\ref{nlsn_up_good}.
Choose  \(n,M\) as above in Theorem \ref{nlsn_up}, where \(\gamma\) is different then before. For both Besov and Sobolev spaces, the error of best trigonometric approximation can be bounded like in Lemma \ref{E_M} but for different constants depending on \(r,\theta\). Bounds of this type can be found in \cite[Lemma 2]{MPU25_1}. The fact that the constants are not given there for all regimes is not a hurdle since \(M\) enters only in a logarithmic way. Then we get from Theorem  \ref{samp_A1} together with Theorem \ref{sigma_m_Besov_in_Wiener_q<theta} (with \(\eta =1\)) 
    \begin{align}
        \begin{split}   \varrho_m\big(S^{r}_{p,\theta}B\big)_{L_q} & \leq C n^{1/2-1/q}\left(n^{-1/2} \sigma_n\big(S^{r}_{p,\theta}B\big)_{\ca} + E_{[-M,M]^d}\big(S^{r}_{p,\theta}B\big)_{L_\infty}\right) \\
    & \leq C n^{1/2-1/q} (n^{-(r-{\frac{1}{p}+\frac{1}{2}})} \log^\ast(n)^{(d-1)(r+1-\frac{1}{p}-\frac{1}{\theta})} + M^{-\gamma}) \\
    & \leq  C n^{-(r-{\frac{1}{p}+\frac{1}{q}})} \log^\ast(n)^{(d-1)(r+1-\frac{1}{p}-\frac{1}{\theta})}\\
    & \leq  C m^{-(r-\frac{1}{p}+\frac{1}{q})} \log^\ast(m)^{(d-1)(r+1-\frac{1}{p}-\frac{1}{\theta}) + 3(r-{\frac{1}{p}+\frac{1}{q}})}.
        \end{split}
    \end{align}

To prove \eqref{samp_SW_above} we completely repeat the steps of the proof of \eqref{samp_SB_above} with the application of Theorem~\ref{WinA_q<p} (with \(\eta =1\)) instead of Theorem~\ref{sigma_m_Besov_in_Wiener_q<theta}.
\end{proof}

  In the above result, just as described in Remark \ref{sqrtlog}, we managed to  {remove} the additional \(\sqrt{\log^\ast(m)}\) term, that would have appeared when using Proposition \ref{samp_Lp}, instead of Theorem \ref{samp_A1}. In the case $q=2$ the bound in Corollary \ref{varrho_for_SB_SW_above} has already been obtained by Dai, Temlyakov \cite{DaTe24} using a different technique involving ``weak orthogonal matching pursuit'' {\tt (WOMP)} as a recovery method. The estimates obtained here together with the results from \cite{MPU25_1} show that this behavior is also possible using  {\tt (rLasso)} as recovery procedures.  
\vspace{0.5cm}

\paragraph{Acknowledgments}
   The first named author is supported by the ESF, being co-financed by the European Union and from tax revenues on the basis of the budget adopted by the Saxonian State Parliament. The second named author is supported 
   by the Philipp Schwartz Initiative of the Alexander von Humboldt Foundation. 
  The third named author is supported by the German Research Foundation (DFG 403/4-1). The authors would like to thank T. Sommerfeld for helpful comments on the subject and presentation of the paper. We would also like to thank W. Sickel for pointing out to us the references \cite{DPW24}, \cite{NNS22} and \cite{NN22}
    {as well as A. Shidlich for pointing out the references \cite{Stepanets_UMZh_2001_N8}, \cite{Stepanets_MAT_2005} and \cite{Gao_JAT_2010}.}
   { Finally, we would like to thank the (anonymous) reviewers for their very thorough and helpful comments. }

\end{document}